\documentclass[11pt]{amsart}
\usepackage{amssymb, graphicx, color}
\usepackage{amsfonts}
\usepackage{amsthm}
\usepackage{enumerate}
\usepackage[mathscr]{eucal}
\usepackage{graphicx}

\usepackage[pdftex, bookmarksnumbered, bookmarksopen, colorlinks, citecolor=blue, linkcolor=blue]{hyperref}

\allowdisplaybreaks

\newtheorem{thm}{Theorem}[section]
\newtheorem{prop}[thm]{Proposition}
\newtheorem{lemma}[thm]{Lemma}

\theoremstyle{definition}

\theoremstyle{remark}
\newtheorem{remark}[thm]{Remark}

\newcommand{\R}{{\mathbb R}}
\newcommand{\N}{{\mathbb N}}
\newcommand{\Z}{{\mathbb Z}}
\newcommand{\C}{{\mathbb C}}

\newcommand{\M}{{\mathcal M}}

\newcommand{\n}{{\mathcal N}}
\newcommand{\bs}{\begin{split}}
\newcommand{\es}{\end{split}}

\newcommand{\tr}{\text{tr}}
\newcommand{\norm}[2]{\|#1\|_{#2}}
\newcommand{\id}{\text{id}}
\newcommand{\la}{\mathcal{L}}
\newcommand{\ten}{\otimes}

\begin{document}
\setcounter{page}{1}	
\title[Noncommutative spherical means over cyclic groups]
{Dimension-free maximal inequalities for noncommutative spherical means over cyclic groups}	



\author[L. Gao]{Li Gao}
\address{
	School of Mathematics and Statistics\\
	Wuhan University\\
	Wuhan 430072\\
	China}

\email{gao.li@whu.edu.cn}
\author[B. Xu]{Bang Xu}
\address{
	School of Mathematical Sciences\\
	Xiamen University\\
	Xiamen 361005\\
	China}

\email{bangxu@xmu.edu.cn}



\begin{abstract}
In this paper, we establish dimension-free $L_p$-estimates for operator-valued maximal spherical means over cyclic groups $\Z_{m+1}^d$ for all $p>1$ and $m\geq1$. The key ingredient is a noncommutative extension of the spectral technique developed by Nevo and Stein. As an application, we obtain a noncommutative spherical maximal inequality for automorphism actions on von Neumann algebras, along with several concrete examples.
%
\end{abstract}

\maketitle

\section{Introduction}

Let $B$ be a symmetric convex body in $d$-dimensional Euclidean space $\R^d$. For a locally integrable function $f$, the Hardy-Littlewood maximal function is defined by
$$M_Bf(v)=\sup_{t>0}\frac{1}{|B|}\int_B|f(v+tu)|du,$$
where $|B|$ denotes the volume of $B$. For $1<p<\infty$, let $C_p(B)$ be the optimal
constant such that
\begin{equation*}\label{euclidean}
	\|M_Bf\|_p\leq C_p(B)\|f\|_p,\ \forall f\in L_p(\R^d);
\end{equation*}
while $p=1$, $C_1(B)$ is taken to satisfy the weak type $(1,1)$ inequality,
\begin{equation*}\label{euclidean1}
	\|M_Bf\|_{1,\infty}\leq C_1(B)\|f\|_1, \ \forall f\in L_1(\R^d).
\end{equation*}
By applying the theory of spherical maximal functions, Stein \cite{St1} proved that  for  the $d$-dimensional Euclidean ball $B=B^d_2$ and for all $p>1$, the constant $C_p(B^d_2)$ is bounded and independent of the dimension $d$. For $p=1$, Stein and Str\"omberg \cite{SS} later established the bound $C_1(B)\lesssim d\log d$. Subsequent work by Bourgain \cite{Bou86} and Carbery \cite{Car86} extended these results, proving dimension-free bounds for $C_p(B)$ for general symmetric convex $B$ for $p>\frac{3}{2}$.
For specific classes of symmetric convex bodies, such as the $d$-dimensional $\ell_q$-ball $B_q^d$,  M\"uller \cite{Mu90} proved the dimension-free bounds for $C_p(B_q^d)$ for all $p>1$ and $1\leq q<\infty$. In the case of the cube $B_\infty^d=[-1,1]^d$, Bourgain \cite{Bou14} showed that the dimension-free bounds of $C_p(B_\infty^d)$ hold for all $1<p<\infty$. Furthermore, Aldaz \cite{Al11} proved that $C_1(B_\infty^d)$ grows to infinity as $d\to \infty$.

 In recent years, dimension-free estimates have been extensively studied for discrete Hardy-Littlewood maximal operators. For $m\geq1$, let $\Z^d_{m+1}$ denote the $d$-dimensional cyclic group equipped with the $\ell_0$ metric (Hamming distance),
$$|u|=|\{1\leq i\leq d:u(i)\neq0\}|$$
for $u=(u(1),\cdot\cdot\cdot,u(d))\in\Z^d_{m+1}$.
For each $1\leq k\leq d$, let
$$\sigma_k:=\frac{1}{|\{|u|=k\}|}\chi_{\{|u|=k\}}=\frac{1}{\binom{d}{k}m^k}\chi_{\{|u|=k\}}$$
denote the $L_1$-normalized indicator function of the $k$-sphere. The associated spherical maximal operator $M$ is given by
$$M(f)(s)=\sup_{1\leq k\leq d}|\sigma_k\ast f(s)|, s\in \Z^d_{m+1}$$
for $f:\Z^d_{m+1}\rightarrow\C$, where the convolution is defined as
$$\sigma_k\ast f(s)=\sum_{u\in\Z^d_{m+1}}\sigma_k(u)f(s-u).$$
The dimension-free boundedness of the maximal operator $M$ was first investigated by Harrow, Kolla and Schulman \cite{HKS}, who established its dimension-free $L_2$-boundedness on the hypercube $\Z^d_{2}$. Krause \cite{K1} later extended this result to all $p>1$. Subsequently, the dimension-free $L_p$-boundedness of $M$ was obtained for all cyclic groups $\Z^d_{m+1}$ with $m\geq1$ in \cite{GKK}. More precisely, Greenblatt, Kolla and Krause \cite{GKK} proved that
	for any $p>1$ and any $m\geq1$, there is a constant $C_{p,m}$ depends only on $p$ and $m$ such that
	\begin{equation}\label{euclidean3}
		\|M(f)\|_{L_p(\Z_{m+1}^d)}\leq C_{p,m}\|f\|_{L_p(\Z_{m+1}^d)},\ \forall f\in L_p(\Z_{m+1}^d).
	\end{equation}
For recent advances in dimension-free estimates for  $M$ on $\Z^d$, we refer the readers to  \cite{NW25} and the references therein.

In the area of noncommutative analysis, remarkable progress has been made in terms of maximal inequalities, for applications in noncommtative ergodic theory and noncommutative harmonic analysis. Notable examples include the noncommutative versions of maximal ergodic inequalities \cite{CHW,HLW,JX}, Hardy-Littlewood maximal inequalities \cite{M}, Stein's spherical maximal inequalities \cite{CH,Hong,LLLW} etc. Despite the significant advances, the investigation of dimension-free bounds for noncommutative maximal inequalities remains largely unexplored, with only limited results in the literature.

The purpose of this paper is to establish the dimension-free $L_p$-bound (\ref{euclidean3}) to the operator-valued setting, which then implies the corresponding noncommutative maximal inequalities. Let $\mathcal{M}$ be a von Neumann algebra equipped with a normal semifinite faithful trace $\tau$.  The noncommutative $L_p$-spaces associated to the pair $(\M,\tau)$ is denoted by  $L_p(\M)$. For each integer $k$, define the convolution operator
\begin{equation}\label{discrete2}
	T_kf(s)=\sigma_k\ast f(s)=\frac{1}{\binom{d}{k}m^k}\sum_{u:|u|=k}f(s-u),
\end{equation}
for operator-valued functions $f:\Z^d_{m+1}\rightarrow \mathcal{M}$. 
Let $\mathcal{N}= L_\infty(\Z^d_{m+1},\M)\cong L_\infty(\Z^d_{m+1})\overline{\otimes}\M $ be the von Neumann algebra
tensor product equipped with the tensor trace $\varphi=\mu\otimes\tau$, where $\mu$ is the Haar measure on $\Z^d_{m+1}$. Note that for $0<p<\infty$,  $L_p(\mathcal{N})$ can be identified as the Bochner $L_p$-space $L_p(\Z^d_{m+1};L_p(\M))$.
The following is our first main result and we refer to the next section for the precise definition of $L_p(\M;\ell_\infty)$.
\begin{thm}\label{main2}
Let $(T_k)_{1\leq k\leq d}$ be the convolution operators defined as (\ref{discrete2}).	Then for any $p>1$ and any operator-valued functions $f\in L_p(\mathcal{N})$, 
	$$\|(T_kf)_{1\leq k\leq d}\|_{L_p(\mathcal{N};\ell_\infty)}\leq C_{p,m}\|f\|_{p},$$
where constant $C_{p,m}$ depends only on $p$ and $m$ (independent of $d$).
\end{thm}
\begin{remark}
	As noted in \cite[Remark 1.6]{GKK}, the region $p> 1$ is optimal in the sense that the optimal constant in the weak $(1,1)$ bound grows at least as fast as $\sqrt{d}$:
	$$\sup_{f\neq0}\frac{\|Mf\|_{L_{1,\infty}(\Z_{m+1}^d)}}{\|f\|_{L_1(\Z_{m+1}^d)}}\geq c\sqrt{d}.$$
\end{remark}

An immediate difficulty in proving Theorem \ref{main2} is the fact that the classical maximal function $Mf= \sup_{1\le k\le d} |T_k f|$ does not make sense in general when $(T_kf)$ is a sequence of operators (see \cite{JX}). Indeed, the right formulation of noncommutative $L_p$-maximal  inequalities was achieved via the introduction of vector-valued noncommutative $L_p$-spaces $L_p(\M;\ell_\infty)$, which play the roles of the $L_p$-norm of maximal function $Mf$ in the noncommutative setting \cite{J1}.

As an application, we establish the noncommutative spherical maximal inequality for automorphism actions of $\Z_{m+1}^d$ on von Neumann algebras. Recall that for a von Neumann algebra $\M$, we say that $\alpha:\Z_{m+1}^d\rightarrow \text{Aut}(\mathcal{M})$ is an action if for each $u$, $\alpha_u:\mathcal{M}
\to \mathcal{M}$ is a $*$-preserving automorphism; and for all $s,u\in \Z_{m+1}^d$, $\alpha_s\circ\alpha_u=\alpha_{s+u}$. If in additional, $\tau\circ \alpha_u=\tau$ for all $u\in \Z_{m+1}^d$, we say $\alpha$ is a $\tau$-preserving action or an action of $\tau$-preserving automorphisms, denoted as $\alpha \curvearrowright (\mathcal{M},\tau)$.
For each $1\leq k\leq d$, define
	\begin{equation}\label{discrete11}
M_kx=\frac{1}{\binom{d}{k}m^k}\sum_{u:|u|=k}\alpha_ux,\ x\in\M.
\end{equation}

Theorem \ref{main2}, together with the noncommutative Calder\'on transference principle developed in \cite{HLW}, implies the following result. Readers may refer to Section 2 for the definition of $\ell^c_\infty$-valued $L_p$-space $L_p(\M;\ell^c_\infty)$.
\begin{thm}\label{main1}
 Let $\alpha$ be an $\tau$-preserving action of $\Z_{m+1}^d$ on a semifinite von Neumann algebra $(\M,\tau)$. Let $(M_k)_{1\leq k\leq d}$ be the  defined as (\ref{discrete11}).	Then for $p>1$, 
	$$\|(M_kx)_{1\leq k\leq d}\|_{L_p(\M;\ell_\infty)}\leq C_{p,m}\|x\|_{L_p(\M)},\ \forall x\in L_p(\M);$$
	and for $p>2$,
	$$\|(M_kx)_{1\leq k\leq d}\|_{L_p(\M;\ell^c_\infty)}\leq \sqrt{C_{p,m}}\|x\|_{L_p(\M)},\ \forall x\in L_p(\M),$$
	where the constant $C_{p,m}$ depends only on $p$ and $m$.
\end{thm}
Theorem \ref{main1} can be rephrased as a maximal inequality for multiple ergodic averages. Let $(L_i)_{1\leq i\leq d}$ be  a commuting family (i.e. $L_iL_j=L_jL_i$) of $\tau$-preserving automorphisms on $\M$ such that $L_i^{m+1}=\id$. For each positive integer $k$, define
\begin{equation}\label{discrete3}
	N_{k}x=\frac{1}{\binom{d}{k}m^k}\sum_{n:|n|=k}L^nx,\ x\in\M.
\end{equation}
Here $n=(n_1,\cdot\cdot\cdot,n_d)$ and $L^n=L_1^{n_1}L_2^{n_2}\cdots L_d^{n_d}$. Then for $p>1$,
			$$\|(N_kx)_{1\leq k\leq d}\|_{L_p(\M;\ell_\infty)}\leq C_{p,m}\|x\|_{L_p(\M)},\ \forall x\in L_p(\M);$$
and $p>2$,
$$\|(N_kx)_{1\leq k\leq d}\|_{L_p(\M;\ell^c_\infty)}\leq \sqrt{C_{p,m}}\|x\|_{L_p(\M)},\ \forall x\in L_p(\M).$$


We conclude this introduction by outlining the structure of the paper. Section 2 reviews some basic background on noncommutative $L_p$-spaces. The proof of Theorem \ref{main2} are presented in Sections 3 and 4. Theorem \ref{main1} is proved in Section 5, and several illustrative examples of its applications are provided in Section 6. 


\section{Preliminaries}
\subsection{Noncommutative $L_p$-spaces}
Let $\M$ be a von Neumann algebra equipped with a normal semifinite faithful trace $\tau:\M_+\rightarrow [0,\infty]$.
Let $\mathcal{S_{\M+}}$ be the set of all positive elements $x\in\M$ such that $\tau(\mathrm{supp}(x))<\infty$, where  $\mathrm{supp}(x)$ denotes the support projection of  $x$ and $\mathcal{S}_{\M}$ be the linear span of $\mathcal{S_{\M+}}$. Then $\mathcal{S}_{\M}$ forms a $w^{*}$-dense $\ast$-subalgebra of $\M$. For $1\leq p<\infty$, the noncommutative $L_p$-space $L_{p}(\M,\tau)$ (in short, $L_p(\mathcal{M})$) is the completion of $\mathcal{S}_{\M}$ under the norm,
$$\|x\|_{p}=[\tau(|x|^p)]^{1/p}\ \ , x\in\mathcal{S}_{\M},$$
where $|x|=(x^{\ast}x)^{\frac{1}{2}}$. By convention,  $L_{\infty}(\M) = \M$ equipped with the operator norm. We write $L^+_{p}(\M)$ the positive cone of $L_{p}(\M)$. We refer to \cite{FK,P2} for further information on noncommutative $L_p$-spaces.

\subsection{Vector-valued noncommutative $L_p$-spaces}\label{tracial}
We refer to \cite{J1,JX} for the references on vector-valued noncommutative $L_p$-spaces. Let $I$ be an index set.
Given $1\le p\leq\infty$, the space $L_p(\mathcal {M};\ell_\infty(I))$ consists of all families $(x_n)_{n\in I}$ in $L_p(\mathcal {M})$ that admit a factorization $x_n=ay_nb$ with $a,b\in L_{2p}( \mathcal {M})$  and $(y_n)_{n\in I}\subset \mathcal {M}$, equipped with the norm
$$\|(x_n)_{n\in I}\|_{L_p(\mathcal {M};\ell_\infty(I))}=\inf\left\{\big\|a\big\|_{{2p}}\sup_{n\in I}\big\|y_n\big\|_\infty\big\|b\big\|_{{2p}}\right\},$$
where the infimum is taken over all possible factorizations $x_n=ay_nb, n\in I$.
Following standard convention, this norm is often denoted concisely as $\|{\sup^+_{n\in I}}x_n\|_p$. 
\begin{remark}\label{rk:MaxFunct}\rm
	For a sequence of positive operators $(x_n)_{n\in I}\subset L_p^+(\mathcal {M})$, $x=(x_n)_{n\in I}$ in $L_p(\mathcal {M};\ell_\infty(I))$ iff there exists $a\in L^+_{p}(\M)$
	such that $x_n\leq a$ for all $n\in I$. Moreover,
	$$\|x\|_{L_p(\mathcal {M}; \ell_\infty(I))}=\inf\Big\{\|a\|_p: a\in L^+_{p}(\M)\ \mbox{such that}\ x_n\leq a
	,\ \forall n\in I\Big\}.$$
\end{remark}

\noindent In the sequel, we omit the index set $I$ when it will not cause confusions.

As established in \cite{J1}, for every $p>1$, the predual space of $L_{p}(\mathcal {M};\ell_\infty)$ is $L_{p'}(\mathcal {M};\ell_1)$, where $p'=\frac{p}{p-1}$ is the conjugate index of $p$. Given $1\leq p\leq\infty$, the space
$L_{p}(\mathcal {M};\ell_1)$ consists of all sequences $x=(x_n)_{n}$ in $L_{p}(\mathcal {M})$ admitting a factorization of the form
$$x_n=\sum_{k}u_{kn}^{\ast}v_{kn}, \quad \forall n $$
for two families $(u_{kn})_{k,n}$ and $(v_{kn})_{k,n}$ in $L_{2p}(\mathcal {M})$ such that
$$\sum_{k,n}u^{\ast}_{kn}u_{kn}\in L_{p}(\mathcal {M})\;\; \text{and}\;\; \sum_{k,n}v^{\ast}_{kn}v_{kn}\in L_{p}(\mathcal {M}).$$
The norm of $x$ in $L_{p}(\mathcal {M};\ell_1)$ is defined by
$$\|x\|_{L_{p}(\mathcal {M};\ell_1)}=\inf \|\sum_{k,n}u^{\ast}_{kn}u_{kn}\|_{p}^{1/2}\|\sum_{k,n}v^{\ast}_{kn}v_{kn}\|_{p}^{1/2},$$
where the infimum runs over all possible decompositions $x_n=\sum_{k}u_{kn}^{\ast}v_{kn}$.
The duality between $L_{p}(\mathcal {M};\ell_{\infty})$ and $L_{p'}(\mathcal {M};\ell_1)$ is given by
$$\langle x,\;y\rangle=\sum_n\tau(x_ny_n),\quad x=(x_n)_{n}\in L_{p}(\mathcal {M};\ell_\infty),\;y=(y_n)_{n}\in L_{p'}(\mathcal {M};\ell_1).$$
\begin{remark}\label{rk:MaxFunct1}\rm
	A positive sequence $x=(x_n)_{n}$ belongs to  $L_p(\mathcal {M};\ell_1)$ iff $\sum_nx_n\in L_{p}(\M)$; and in this case,
	$$\|x\|_{L_p(\mathcal {M}; \ell_1)}=\big\|\sum_nx_n\big\|_p.$$
\end{remark}
We collect some important properties of these two vector-valued noncommutative $L_p$-spaces. See \cite{JX} for the detailed proofs.
\begin{prop}\label{JunXu}
The following statements hold.
		\begin{itemize}
		\item[(i)] Each element in the unit ball of $L_p(\M;\ell_\infty)$ (resp. $L_p(\M;\ell_1)$) can be written as a sum of
		sixteen (resp. eight) positive elements in the same ball.
		\item[(ii)] Let $x=(x_n)$ be a positive sequence in $L_p(\M;\ell_\infty)$. Then for any $1\leq p\leq\infty$,
		$$\|{\sup_{n}}^+x_n\|_p=\sup\Big\{\sum_n\tau(x_ny_n):y_n\in L_{p^\prime}(\M)\ \mbox{and}\ \big\|\sum_ny_n\big\|_{p^\prime}\leq1\Big\}.$$
		In particular, for two positive sequences $(x_n)\le (y_n)$, $\|{\sup_{n}}^+x_n\|_p\le \|{\sup_{n}}^+y_n\|_p.$
		\item[(iii)] Let $x=(x_n)$ be a  sequence in $L_p(\M;\ell_\infty)$ and $(z_{n,k})_{n,k}\in\C$. Then for any $1\leq p\leq\infty$, $$\big\|{\sup_{n}}^+\sum_kz_{n,k}x_k\big\|_p\leq\sup_n(\sum_k|z_{n,k}|)\|{\sup_{n}}^+x_n\|_p.$$
		
		\item[(iv)] Let $1\leq p_0<p_1\leq \infty$ and $0<\theta<1$. We have the following complex interpolation relation,
		$$L_p(\M;\ell_1)=\big(L_{p_0}(\M;\ell_1),L_{p_1}(\M;\ell_1)\big)_\theta\ \ \text{isometrically},$$
		where $\frac{1}{p}=\frac{1-\theta}{p_0}+\frac{\theta}{p_1}$.
	\end{itemize}	
\end{prop}
We then turn to the definition of  $L_p(\M;\ell^c_\infty)$. Let $2\leq p\leq\infty$, the space $L_p(\M;\ell^c_\infty)$ is defined as the family of all $(x_n)_n$ in $L_p(\M)$ for which there are an $a\in L_p(\M)$ and $(z_n)_n\subseteq L_\infty(\M)$ such that
$$x_n=z_na\ \ \mbox{and}\ \ \sup_n\|z_n\|_\infty<\infty.$$
The norm $\|(x_n)_n\|_{L_p(\M;\ell^c_\infty)}$ is then defined as  to be the infimum  $\inf\{\sup_n\|z_n\|_\infty\|a\|_p\}$ over all
factorizations of $(x_n)_n=(z_n)a$ as above. Note that $(x_n)_n\in L_p(\M;\ell^c_\infty)$ iff $(x^*_nx_n)_n\in L_{p/2}(\M;\ell_\infty)$. 

Finally, we introduce the $\ell_\infty$-valued noncommutative weak $L_p$-spaces $\mathcal{L}_{p,\infty}(\M;\ell_{\infty})$ (see e.g. \cite{HLX}).
Given a sequence $(x_n)_{n}$
in $L_p(\mathcal {M})$ with $1\leq p<\infty$, we define
\begin{align*}
	\|(x_n)_{n}\|_{\mathcal{L}_{p,\infty}(\M;\ell_{\infty})}
	=\sup_{\mathcal{L}>0} \ \mathcal{L}\inf_{}
	\Big\{\big(\tau(e^{\perp})\big)^{\frac{1}{p}}: e\in\mathcal{P}(\M) \text{ s.t. } \|ex_ke\|_{\infty}\leq\mathcal{L}\ \mbox{for\ all}\ k\Big\},
\end{align*}where $\mathcal{P}(\M)$ is the set of all projections in $\M$. The space $\mathcal{L}_{p,\infty}(\M;\ell_{\infty})$ consists of all sequences $x=(x_n)_{n}$ for which this quasi-norm is finite.

\subsection{Noise operators} Recall that the dual group of $\Z_{m+1}^d$ is itself $\Z_{m+1}^d$. In the following, we will denote points in the group by lowercase letters, and frequencies in the dual group by capital letters.
For each $S,u\in \Z_{m+1}^d$,  we define the $L_2$-normalized character
$$\chi_{S}(u)=\frac{1}{(m+1)^{d/2}}\xi_m^{S\cdot u}=\frac{1}{(m+1)^{d/2}}\prod_{i=1}^d\xi_m^{S(i)u(i)},$$
where $\xi_m=e^{2\pi i/(m+1)}$ is a primitive $(m+1)$-th root of unity and $S\cdot u=\sum_{i=1}^dS(i)u(i)$. A operator-valued function $f:\Z_{m+1}^d\rightarrow \M$ admits a Fourier expansion of the form:
$$f(u)=\sum_{S\in\Z_{m+1}^d}\widehat{f}(S)\chi_{S}(u),$$
where each Fourier coefficient is given by
$$\widehat{f}(S)=\sum_{u\in\Z_{m+1}^d}f(u)\overline{\chi_{S}(u)}.$$

For each $t>0$, the noise operator $N_t$ is defined by via its Fourier expansion
\begin{align}\label{noise2}
	N_tf(u)=\sum_{S\in\Z_{m+1}^d}e^{-t|S|}\widehat{f}(S)\chi_{S}(u),
\end{align}
and thus the family of operators $(N_t)_{t\geq 0}$ form a semigroup. For each $T>0$, define
\begin{align}\label{noise4}
	H_Tf=\frac{1}{T}\int_0^TN_tfdt.
\end{align}
The boundedness theory for the family $(H_T)_{T\geq0}$ was first established by Lance \cite{Lan76} for the case $p=\infty$, and later by Yeadon \cite{Y1} for $p=1$. However, the case $1<p<\infty$ remained open for thirty years until it was finally resolved by Junge and Xu \cite{JX} through the development of noncommutative Marcinkiewicz interpolation. Junge-Xu's result has since been widely applied in noncommutative ergodic theory and noncommutative harmonic analysis. We state their results as the following theorem.
\begin{thm}[\cite{Lan76,JX,Y1}]\label{maximal1}
Let $\mathcal{N}=L_\infty(\Z^d_{m+1})\bar{\otimes}\M$.
The following assertions hold with a positive constant
$C_{p,m}$ depending only on $p$ and $m$:
	\begin{itemize}
		\item[(i)] for $p=1$ and any operator-valued function $f\in L_1(\mathcal{N})$,
		$$\|(H_Tf)_{T>0}\|_{\mathcal{L}_{1,\infty}(\mathcal{N};\ell_\infty)}\leq C_{1,m}\|f\|_{L_1(\mathcal{N})};$$
		
		\item[(ii)] for $1<p\leq \infty$ and any operator-valued function $f\in L_p(\mathcal{N})$, $$\|(H_Tf)_{T>0}\|_{L_{p}(\mathcal{N};\ell_\infty)}\leq C_{p,m}\|f\|_{L_p(\mathcal{N})}\ \  .$$	
	\end{itemize}	
\end{thm}

\section{The smooth spherical maximal operators}
In this section, we study the boundedness theory of two smooth spherical maximal operators.
By triangle's inequality in $L_p(\mathcal{N};\ell_\infty)$,
\begin{align*}
	\|(T_kf)_{1\leq k\leq d}\|_{L_p(\mathcal{N};\ell_\infty)} &\leq
	\|(T_kf)_{1\leq k\leq \frac{md}{m+1}}\|_{L_p(\mathcal{N};\ell_\infty)}+\|(T_kf)_{\frac{md}{m+1}\leq k\leq d}\|_{L_p(\mathcal{N};\ell_\infty)}
	\\
	&= \|(T_kf)_{1\leq k\leq \frac{md}{m+1}}\|_{L_p(\mathcal{N};\ell_\infty)}+\|(T_{d-k}f)_{1\leq k\leq \frac{d}{m+1}}\|_{L_p(\mathcal{N};\ell_\infty)},
\end{align*}
where the convolution operators $T_k, 1\le k\le d$ are defined in (\ref{discrete2}). Therefore,
to establish Theorem \ref{main2}, it suffices to  estimate the aforementioned local term $(T_kf)_{1\leq k\leq \frac{md}{m+1}}$ and distant terms $(T_{d-k}f)_{1\leq k\leq \frac{d}{m+1}}$. For this purpose, we introduce the corresponding regularized operators for both local and distant terms. For $K\geq0$, define
\begin{align*}&T_K^M:= \frac{1}{K+1}\sum_{k\leq K} T_k\ ,\  T_K^Mf=\frac{1}{K+1}\sum_{k\leq K}\sigma_k\ast f , \\ \ \ &T_K^D:=\frac{1}{K+1}\sum_{k\leq K} T_{d-k} \ ,\ T_K^Df=\frac{1}{K+1}\sum_{k\leq K}\sigma_{d-k}\ast f.\end{align*}
\begin{prop}\label{smooth1}
	Let $1\leq p\leq\infty$. 
	There exists a constant $C_{p,m}$ depending only on $p$ and $m$ such that
	\begin{itemize}
		\item[(i)] for $p=1$ and any $f\in L_1(\mathcal{N})=L_1(\mathbb{Z}_{m+1}^d; L_1(\mathcal{M}))$, 
\begin{align*}&\|(T_K^Mf)_{0\leq K\leq \frac{md}{m+1}}\|_{\mathcal{L}_{1,\infty}(\mathcal{N};\ell_\infty)}\leq C_{1,m}\|f\|_{L_1(\mathcal{N})}\ ,\\
&\|(T_K^Df)_{0\leq K\leq \frac{d}{m+1}}\|_{\mathcal{L}_{1,\infty}(\mathcal{N};\ell_\infty)}\leq C_{1,m}\|f\|_{L_1(\mathcal{N})}\ ;
\end{align*}
		\item[(ii)] for $1<p\leq \infty$ and any $f\in L_p(\mathcal{N})=L_p(\mathbb{Z}_{m+1}^d; L_p(\mathcal{M}))$, 
\begin{align*}&\|(T_K^Mf)_{0\leq K\leq \frac{md}{m+1}}\|_{L_{p}(\mathcal{N};\ell_\infty)}\leq C_{p,m}\|f\|_{L_p(\mathcal{N})}\ ,\\
&\|(T_K^Df)_{0\leq K\leq \frac{d}{m+1}}\|_{L_{p}(\mathcal{N};\ell_\infty)}\leq C_{p,m}\|f\|_{L_p(\mathcal{N})}\ .
\end{align*}
	\end{itemize}
\end{prop}

The proof of Proposition \ref{smooth1} relys on Theorem \ref{maximal1}. To see this,
given any fixed parameter $\eta\in (0,1)$, we define a probability measure $\mu_\eta$ on $\Z_{m+1}$ by $\mu_\eta(w)=1-\eta$ if $w=0$ and $\mu_\eta(w)=\frac{\eta}{m}$ otherwise. 
This induces the product measure on $\Z_{m+1}^d$: for $u\in\Z_{m+1}^d$,
$$\mu^d_\eta(u)=\Big(\frac{\eta}{m}\Big)^{|u|}(1-\eta)^{d-|u|}.$$
It was proved from \cite[Lemma 4.6]{GKK} that
\begin{align}\label{noise1} \hat{\mu}_\eta(S)=\Big(1-\frac{(m+1)\eta}{m}\Big)^{|S|}\ , \ f\ast \mu^d_\eta(u)=\sum_{S\in\Z_{m+1}^d}\Big(1-\frac{(m+1)\eta}{m}\Big)^{|S|}\widehat{f}(S)\chi_{S}(u).
\end{align}
where $\hat{\mu}^d_\eta$ denotes the Fourier transform of $\mu^d_\eta$.
For $0<P\leq \frac{m}{m+1}$, denote 
\begin{align}\label{noise3}
\widetilde{N}_\eta f=f\ast \mu^d_\eta\ ,\ 	J_Pf:=\frac{1}{P}\int_{0}^P\widetilde{N}_\eta fd\eta.
\end{align}
Note that $\widetilde{N}_\eta=N_{(1-\frac{(m+1)\eta}{m})}$ is a reparametrization of the noise operator.
\begin{lemma}\label{ergodic}
There exists a constant $C_{p,m}$ depending only on $p$ and $m$ such that
	\begin{itemize}
		\item[(i)] for $p=1$,
$$\|(J_Pf)_{0<P\leq \frac{m}{m+1}}\|_{\mathcal{L}_{1,\infty}(\mathcal{N};\ell_\infty)}\leq C_{1,m}\|f\|_{1}\ \ \forall f\in L_1(\mathcal{N});$$
		
		\item[(ii)] for $1<p\leq \infty$, $$\|(J_Pf)_{0<P\leq \frac{m}{m+1}}\|_{L_{p}(\mathcal{N};\ell_\infty)}\leq C_{p,m}\|f\|_{p}\ \ \forall f\in L_p(\mathcal{N}).$$
	\end{itemize}		
\end{lemma}
\begin{proof}
	In the proof, we may assume that $f$ is positive. Indeed, by decomposing $f=f_1-f_2+i(f_3-f_4)$ with positive $f_j$ and $\|f_j\|_p\leq\|f\|_p$ for $j\in\{1,2,3,4\}$ and all $p\geq1$, we can use (quasi-)triangle's inequality to restrict the proof to the case of positive operators.
	
	(i) Following \cite[Proposition 4.7]{GKK},
	 we use the short notation $\rho_m:=\frac{m}{m+1}$ and define the measure $$\nu_P=\Big\{\frac{\rho_m}{P}Te^{-T}\chi_{(0,-\ln[1-P/\rho_m])}\Big\}dT+\Big\{\Big(\frac{\rho_m}{P}-1\Big)\Big(-\ln\Big[1-\frac{P}{\rho_m}\Big]\Big)\Big\}\delta_{-\ln[1-P/\rho_m]},$$
for $0<P\leq \rho_m$, 	where $\delta$ denotes the dirac measure. One can verify that $\nu_P$ has total mass 1. Moreover, it was shown in \cite[Proposition 4.7]{GKK} that
	\begin{align}\label{noise5}
		J_Pf=\int_0^\infty\Big(\frac{1}{T}\int_0^TN_tfdt\Big)d\nu_P(T)=\int_0^\infty H_Tfd\nu_P(T).
	\end{align}
We now appeal to the weak type $(1,1)$ boundedness of $(H_T)_T$ stated in Theorem \ref{maximal1} (i) that
 for any $\mathcal{L}>0$, there exists a projection $e\in\mathcal{N}$ such that $\mu\otimes \tau(1-e)\leq C_{1,m}\frac{\|f\|_1}{\mathcal{L}}$ and
	$$\sup_{0<P\leq \frac{m}{m+1}}\|eJ_Pfe\|_\infty\leq\int_0^\infty\sup_{T>0}\|eH_Tfe\|_\infty d\nu_P(T)\leq\mathcal{L}.$$
	This proves (i).
	
	(ii) The strong type $(p,p)$ estimates follows analogously. Indeed, by Theorem \ref{maximal1} (ii) and Remark \ref{rk:MaxFunct}, there exist $F\in L^+_p(\n)$ and a constant $C_{p,m}$ such that for all $T>0$,
	$$H_Tf\leq F\ \ \mbox{and}\ \ \|F\|_p\leq C_{p,m}\|f\|_p.$$
	It then follows from (\ref{noise5}) that for  all $0<P\leq \rho_m$,
	$$	J_Pf\leq \int_0^\infty Fd\nu_P(T)=F.$$
Therefore, by using Remark \ref{rk:MaxFunct} once more, we obtain
	$$\|\underset{0<P\leq \frac{m}{m+1}}{\text{sup}^+} J_Pf\|_{p}\leq\|F\|_p\leq C_{p,m}\|f\|_{p},$$
which finishes the proof.	
\end{proof}


Now we are ready to prove Proposition \ref{smooth1} .
\begin{proof}[Proof of Proposition \ref{smooth1}.]
	Without loss of generality, we can assume that $f$ is positive. 
By the kernel estimate presented in \cite[Proposition 4.8]{GKK}, for each $K\leq d$, one can choose $P(K)\in (0,\frac{m}{m+1}]$ and an absolute constant $C>0$ such that
	$$\frac{1}{K+1}\sum_{l\leq K}\sigma_l\leq C\frac{1}{P(K)}\int_0^{P(K)}\mu^d_\eta d\eta.$$
As a consequence,
	$$T_K^Mf=\frac{1}{K+1}\sum_{l\leq K}\sigma_l\ast f\leq C\frac{1}{P(K)}\int_0^{P(K)}\widetilde{N}_\eta fd\eta=CJ_{P(K)}f.$$
	The desired bound for $T_K^Mf$ then follows from the corresponding estimates in
	Lemma \ref{ergodic}. 

	For $T^D_K$, by similar kernel estimate in \cite[Proposition 4.8]{GKK}, we have that for any $0\leq L\leq \frac{md}{m+1}$, there exists a constant $C_m$ such that
	$$\frac{1}{\lfloor\frac{L}{m}\rfloor +1}\sum_{l\leq L/m}\sigma_{d-l}\leq C_m\frac{1}{L+1}\sum_{l\leq L}\sigma_l\ast\sigma_d,$$
	where $\lfloor a\rfloor$ denotes the integer part of $a$. Hence, for any $K\leq\frac{d}{m+1}$
	\begin{align}\label{noise6}
		T_K^Df\leq C_m\frac{1}{L+1}\sum_{l\leq L}\sigma_l\ast (\sigma_d\ast f).
	\end{align}
	By the weak $(1,1)$-estimate of $T_K^M$, for any $\mathcal{L}>0$ there exists a projection $e\in\mathcal{N}$ such that
	$$\sup_{0\leq K\leq\frac{d}{m+1}}\|e(T_K^Df)e\|_\infty\leq C_m\sup_{L\leq \frac{md}{m+1}}\|e(T^M_L(\sigma_d\ast f))e\|_\infty\leq C_m\mathcal{L}$$
	and
	$$\mu\otimes\tau(1-e)\leq\frac{\|\sigma_d\ast f\|_1}{\mathcal{L}}\leq C_m\frac{\|f\|_1}{\mathcal{L}},$$
	where we used the Young inequality ($\sigma_d$ is a probability measure)
	\begin{align}\label{noise7}
		\|\sigma_d\ast f\|_p\leq \|f\|_p,\, \ \,  \forall\  1\leq p\leq\infty.
	\end{align}
	Then we combine (\ref{noise6}), the strong type $(p,p)$ estimate for $T_K^M$ and (\ref{noise7}) to obtain
	\begin{align*}
		\big\|\underset{0\leq K\leq\frac{d}{m+1}}{\text{sup}^+} T_K^Df\big\|_{p} &\leq
		C_m \big\|\underset{0\leq L\leq \frac{md}{m+1}}{\text{sup}^+}T^M_L(\sigma_d\ast f)\big\|_{p}
		\\
		&\leq C_{p,m}\|\sigma_d\ast f\|_{p}\leq C_{p,m}\|f\|_{p}.
	\end{align*}
	This proves the strong type $(p,p)$ estimates for $T_K^D$.
\end{proof}
\section{Proof of Theorem \ref{main2}}

This section is devoted to the proof of Theorem \ref{main2}. Let us begin by recalling the decomposition from the previous section:
$$	\|(T_kf)_{1\leq k\leq d}\|_{L_p(\mathcal{N};\ell_\infty)}\leq\|(T_kf)_{1\leq k\leq \frac{md}{m+1}}\|_{L_p(\mathcal{N};\ell_\infty)}+\|(T_{d-k}f)_{1\leq k\leq \frac{d}{m+1}}\|_{L_p(\mathcal{N};\ell_\infty)}.$$
Then it suffices to establish the following two estimates:
	\begin{align}\label{noise11}
	\|(T_kf)_{1\leq k\leq \frac{md}{m+1}}\|_{L_p(\mathcal{N};\ell_\infty)}\leq C_{p,m}\|f\|_p;
\end{align}
\begin{align}\label{noise12}
	\|(T_{d-k}f)_{1\leq k\leq \frac{d}{m+1}}\|_{L_p(\mathcal{N};\ell_\infty)}\leq C_{p,m}\|f\|_p.
\end{align}
The proofs of (\ref{noise11}) and (\ref{noise12}) are inspired by the complex Cesaro sums
 method developed by Nevo and Stein \cite{NS}.
We introduce some necessary notations. Given a  complex number $\alpha$ and a nonnegative integer $n$,  set
$$A^\alpha_n = \frac{ (\alpha + 1)(\alpha + 2) \dots (\alpha + n)}{n!},\  \ A_0^\alpha:= 1, \ A_{-1}^\alpha:=0.$$

The following elementary properties on $A_n^\alpha$ are well-known, see \cite[Lemma 3]{NS}.
\begin{lemma}\label{NS1}
	Let $\alpha=\beta+i\gamma$ be a complex number with $\beta,\gamma\in\R$. There exists $c_\beta>0$ such that
	\begin{itemize}
		\item[(i)] for $\beta>-1$,
		$$c_\beta^{-1}(n+1)^\beta\leq A_n^\beta\leq c_\beta(n+1)^\beta;$$
		
		\item[(ii)] for $\beta>-1$,
		$$|A_n^\beta|\leq|A_n^\alpha|\leq c_\beta e^{2\gamma^2}|A_n^\beta|;$$
		
		\item[(iii)] for $\beta\in\N$,
		$$|(1+n)^\beta A_n^{-\beta+i\gamma}|\leq c_\beta e^{3\gamma^2}.$$
	\end{itemize}		
\end{lemma}
For $\alpha\in\C$ and $n\in\N$, we define
\begin{align*}&S_n^\alpha f= \sum_{k=0}^n A_{n-k}^\alpha T_kf\ , \ M_n^\alpha f= \frac{1}{(n+1)^{\alpha +1}}S_n^\alpha f;\\ \ \ &U_n^\alpha f= \sum_{k=0}^n A_{n-k}^\alpha T_{d-k}f\ , \ N_n^\alpha f= \frac{1}{(n+1)^{\alpha +1}}U_n^\alpha f.\end{align*}

We will establish the following more general estimate.
\begin{thm}\label{generalize1}Let $(\mathcal{M},\tau)$ be a semifinite von Neumann algebra.
	For any $\alpha\in\C$ and $p>1$, there exists a constant $C_{m,p,\alpha}$ depends only on $m,p,\alpha$ such that $f\in L_p(\mathcal{N})$
	\begin{align*}&\|{\underset{0\leq n\leq \frac{md}{m+1}}{\mathrm{sup}^+}}M_n^\alpha(f)\|_{p}\leq C_{m,p,\alpha}\|f\|_{p};\\ &\|{\underset{0\leq n\leq \frac{d}{m+1}}{\mathrm{sup}^+}}N_n^\alpha(f)\|_{p}\leq C_{m,p,\alpha}\|f\|_{p}.\end{align*}
\end{thm}
Indeed, Theorem \ref{generalize1} directly yields estimates (\ref{noise11}) and (\ref{noise12}) through the identifications $M_n^{-1}f = T_nf$ and $N_n^{-1}f = T_{d-n}f$. Our proof strategy for Theorem \ref{generalize1} follows the framework established by Nevo and Stein in \cite[Section 3]{NS}, which we adapt to the noncommutative setting. The proof of Theorem \ref{generalize1} will be divided into the following four lemmas.
\begin{lemma}\label{NS41}
	Let $\alpha=\beta+i\gamma$ with $\beta\geq0$. Then for all $p>1$ and $f\in L_p(\mathcal{N})$,
	$$\|{\underset{0\leq n\leq \frac{md}{m+1}}{\mathrm{sup}^+}}M_n^\alpha(f)\|_{p}\leq C_{m,p,\beta}e^{2\gamma^2}\|f\|_{p};$$
	$$\|{\underset{0\leq n\leq \frac{d}{m+1}}{\mathrm{sup}^+}}N_n^\alpha(f)\|_{p}\leq C_{m,p,\beta}e^{2\gamma^2}\|f\|_{p}.$$
\end{lemma}
\begin{proof}
	Without loss of generality, we assume that $f$ is positive. 
	Let $(g_n)\subset L^+_{p^\prime}(\mathcal{N})$ with $\|\sum_n g_n\|_{L_{p^\prime}(\mathcal{N})}\leq1$. By applying Lemma \ref{NS1}, we have
	\begin{align*}
		|\tau(M_n^\alpha (f)g_n)| &\leq
		(n+1)^{-\beta-1}\sum_{k=0}^n|A_{n-k}^\alpha|\tau(T_kfg_n)
		\\
		&\leq c_\beta e^{2\gamma^2}(n+1)^{-\beta-1}\sum_{k=0}^n(n-k+1)^\beta\tau(T_kfg_n)\\
		&\leq  c_\beta e^{2\gamma^2}(n+1)^{-1}\sum_{k=0}^n\tau(T_kfg_n).
	\end{align*}
	Therefore, recalling that $T_n^Mf=(n+1)^{-1}\sum_{k\leq n}T_kf$, we obtain
	$$\sum_{0\leq n\leq \frac{md}{m+1}}|\tau(M_n^\alpha (f)g_n)|\leq c_\beta e^{2\gamma^2}\sum_{0\leq n\leq \frac{md}{m+1}}\tau(T_n^Mfg_n).$$	
	Taking the supremum over all $(g_n)\subset L^+_{p^\prime}(\mathcal{N})$ with $\|\sum_n g_n\|_{L_{p^\prime}(\mathcal{N})}\leq1$, using Proposition \ref{JunXu} (i)(ii) and Proposition \ref{smooth1} (ii), we deduce the first assertion. The second statement can be proved similarly by using the definition of $T_n^Df$ and Proposition \ref{smooth1} (ii).
\end{proof}
The following $L_2$-estimates are identical to the scalar-valued cases in \cite[Proposition 5.7]{GKK}, so we omit the proof. For an operator $x$, recall that $|x|=(x^*x)^{\frac{1}{2}}$. 
\begin{lemma}\label{L_2}
	Let $t\in\N$. Then there exists constant $C_{t,m}>0$ depending only on $t$ and $m$ such that for all $f\in L^+_2(\mathcal{N})$,
	$$\|R^1_{t,m}f\|_2\leq C_{t,m}\|f\|_2\ \mbox{ and }\ \|R^2_{t,m}f\|_2\leq C_{t,m}\|f\|_2,$$
	where $R^1_{t,m}f$ and $R^2_{t,m}f$ are respectively defined by 
	$$R^1_{t,m}f=\Big(\sum_{0\leq k\leq \frac{md}{m+1}}(k+1)^{2t-1}|S_{k}^{-t-1}f|^2\Big)^{\frac{1}{2}};$$	
		$$R^2_{t,m}f=\Big(\sum_{0\leq k\leq \frac{d}{m+1}}(k+1)^{2t-1}|U_{k}^{-t-1}f|^2\Big)^{\frac{1}{2}}.$$	
\end{lemma}
With Lemma \ref{L_2} in hand, the next step is to prove the maximal inequalities in $L_2(\mathcal{N})$ for $M_n^{-t}$ and $N_n^{-t}$ with $t\in\N$. For positive integers $t$ and $k$, set
 $$B_k^t=k(k-1)\cdots(k-t+1)\ \mbox{for}\; t\le k \quad\mbox{and}\quad
B_k^t=0\ \mbox{for}\; t> k.$$
\begin{lemma}\label{NS4}
Let $t\in\N$. Then there exists $C_{t,m}$ such that for all $f\in L_2(\mathcal{N})$
	$$\|{\underset{0\leq n\leq \frac{md}{m+1}}{\mathrm{sup}^+}}M_n^{-t}(f)\|_{2}\leq C_{t,m}\|f\|_{2};$$
	$$\|{\underset{0\leq n\leq \frac{d}{m+1}}{\mathrm{sup}^+}}N_n^{-t}(f)\|_{2}\leq C_{t,m}\|f\|_{2}.$$
\end{lemma}
\begin{proof}
	We prove the first inequality, as the second follows analogously. Again,
	we can assume that the operator-valued function $f$ is positive. Observe that for any sequence $(a_k)$,
	\begin{align*}\sum_{k=t}^n B_{k}^{t+1}\,(a_k-a_{k-1})=&  B_n^{t+1}\,a_n -
	\,\sum_{k=t}^{n-1} (B_{k+1}^{t+1}-B_{k}^{t+1})\,a_k-B_{t}^{t+1}a_{t-1}\\ =& B_n^{t+1}\,a_n -
	(t+1)\,\sum_{k=t}^{n-1} B_{k}^{t}\,a_k.
\end{align*}
	Applying above formula to $a_k=S_k^{-t}$, we deduce
	\begin{align}\label{noise129}
B_n^{t+1}\,S_n^{-t}=\sum_{k=t}^nB_k^{t+1}\Delta S_k^{-t}+(t+1)\sum_{k=t}^{n-1}B_k^tS_k^{-t},
\end{align}	
where	$\Delta S_k^{-t}=S_k^{-t}-S_{k-1}^{-t}$. 
We claim that the desired maximal inequality is equivalent to
	\begin{align}\label{noise49}
	\Big\|{\underset{0\leq n\leq \frac{md}{m+1}}{\mathrm{sup}^+}}\frac{1}{(1+n)^2}B_n^{t+1}S_n^{-t}f\Big\|_{2}\leq C_t\|f\|_2.
\end{align}	
Indeed, for each positive integer $n$, one has
$$\frac{1}{(1+n)^2}B_n^{t+1}\leq\frac{(1+n)^{t+1}}{(1+n)^2}=(1+n)^{t-1}.$$
On the other hand, observe that
$$\frac{(1+n)^{t+1}}{B_n^{t+1}}=\frac{1+n}{n}\cdot\cdot\cdot\frac{1+n}{n-t}\leq\Big(\frac{1+n}{n-t}\Big)^{t+1}=\Big(1+\frac{1+t}{n-t}\Big)^{t+1}.$$
If $0<t< \frac{n}{2}$, then
$$\frac{1+t}{n-t}<\frac{2t+2}{n}<\frac{2t+2}{2t}\leq2.$$
If $\frac{n}{2}\leq t\leq n-1$, then
$\frac{1+t}{n-t}\leq t+1$. Therefore, there is some constant $C_t$ depending only on $t$ such that $\frac{(1+n)^{t+1}}{B_n^{t+1}}\leq C_t$. This proves the claim.

To prove (\ref{noise49}), we use (\ref{noise129}) to get
	\begin{align*}
	\frac{B_n^{t+1}}{(1+n)^{2}}\,S_n^{-t}f &=
\frac{1}{(1+n)^{2}}\sum_{k=t}^nB_k^{t+1}\Delta S_k^{-t}f+\frac{t+1}{(1+n)^{2}}\sum_{k=t}^{n-1}B_k^tS_k^{-t}f
	\\
	&=: G^1_{n,t}f+G^2_{n,t}f.
\end{align*}	
	By the convexity of the operator-valued function $x\mapsto |x|^2$ (see e.g. \cite[(1.13)]{M}),
	$$|G^1_{n,t}f|\leq	\frac{1}{(1+n)^{2}}\big(\sum_{k=t}^{n}(B_k^{t+1})^2(k+1)^{-2t+1}\big)^{\frac{1}{2}}\Big(\sum_{k=t}^{n}(k+1)^{2t-1}|\Delta S_k^{-t}f|^2\Big)^{\frac{1}{2}}.$$
		Observe that for $k\geq t+1$,
		$$B_k^{t+1}=k(k-1)\cdot\cdot\cdot(k-t)\leq k^{t+1},$$
		which gives that
	$$\frac{1}{(1+n)^{2}}\big(\sum_{k=t}^{n}(B_k^{t+1})^2(k+1)^{-2t+1}\big)^{\frac{1}{2}}\leq\frac{1}{(1+n)^{2}}\big(\sum_{k=t}^{n}k^3\big)^{\frac{1}{2}}\leq1.$$	
		On the other hand, it follows from \cite[Lemma 2]{NS} that $\Delta S_k^{-t}=S_k^{-t-1}$. 
Therefore, we find
	\begin{align*}
		|G^1_{n,t}f|\leq\Big(\sum_{k=n}^{2n}(k+1)^{2t-1}| S_k^{-t-1}f|^2\Big)^{\frac{1}{2}}\leq R^1_{t,m}f.
	\end{align*}

	We now turn to the second term $G^2_{n,t}f$.
	Using the convexity of the operator-valued function $x\mapsto |x|^2$ again, we have
		\begin{align*}
	|G^2_{n,t}f| &\leq
 \frac{t+1}{(1+n)^{2}}\big(\sum_{k=t}^{n-1}k^3\big)^{\frac{1}{2}}\big(\sum_{k=t}^{n-1}k^{-3}(B_k^t)^2|S_k^{-t}f|^2\big)^{\frac{1}{2}}
		\\
		&\leq 	C_t\big(\sum_{k=t}^{n-1}k^{2t-3}|S_k^{-t}f|^2\big)^{\frac{1}{2}}\leq C_tR^1_{t,m}f.
	\end{align*}
	
Now let $(g_n)\subset L^+_2(\mathcal{N})$ with $\|\sum_n g_n\|_{2}\leq1$. Then by applying the estimates of $G^1_{n,t}f$ and $G^2_{n,t}f$ obtained above, we get
	\begin{align*}
	\Big|	\tau\Big(\frac{1}{(1+n)^2}B_n^{t+1}S_n^{-t}fg_n\Big) \Big|
	&\leq \tau(|G^1_{n,t}f|g_n)+\tau(|G^2_{n,t}f|g_n)\\
	&\leq	C_{t}\tau(R^1_{t,m}fg_n).
\end{align*}
Hence, by Lemma \ref{NS4},
$$\Big|	\sum_n\tau\Big(\frac{1}{(1+n)^2}B_n^{t+1}S_n^{-t}fg_n\Big) \Big|\leq C_t\|R^1_{t,m}f\|_2\|\sum_ng_n\|_2\leq C_t\|f\|_2.$$
%
This implies (\ref{noise49}). 
Thus the proof is finished.
\end{proof}

\begin{lemma}\label{NS14}
	Let $t\in\N$ and $\beta\in\R$. Then there exists a constant $C_t$ such that for any $f\in L_2(\mathcal{N})$,
	$$\big\|{\underset{0\leq n\leq \frac{md}{m+1}}{\mathrm{sup}^+}}M_n^{-t+i\beta}f\big\|_{2}\leq C_te^{3\beta^2}\|f\|_{2},$$
		$$\big\|{\underset{0\leq n\leq \frac{d}{m+1}}{\mathrm{sup}^+}}N_n^{-t+i\beta}f\big\|_{2}\leq C_te^{3\beta^2}\|f\|_{2}.$$		
\end{lemma}
\begin{proof}
	We prove the first statement, as the other one can be done similarly.
	Assume that $f$ is positive and $n\leq10t$ first. By the convolution formula stated in \cite[Lemma 2]{NS}, one has
	$$S^{-t+i\beta}_nf=\sum_{k=0}^nA_{n-k}^{i\beta}S_k^{-t-1}f.$$
	Note that $|A_{n-k}^{i\beta}|\leq c_0e^{2\beta^2}$ for all $k\leq n$ by Lemma \ref{NS1} (ii). Then it follows from the triangle inequality and Proposition \ref{JunXu} (iii) that
	\begin{align*}
		\|{\underset{0\leq n\leq \frac{md}{m+1}}{\mathrm{sup}^+}}M_n^{-t+i\beta}f\|_{2} &\leq \sup_{0\leq k\leq n}|A_{n-k}^{i\beta}|\sum_{k=0}^n\frac{(1+n)^{t-1}}{(1+k)^t}\|{\underset{0\leq k\leq \frac{md}{m+1}}{\mathrm{sup}^+}}M_k^{-t-1}f\|_{2}\\
		&\leq c_te^{2\beta^2}\|f\|_2,
	\end{align*}
	where the second inequality holds since $n\leq10t$ and Lemma \ref{NS4}.
	
Let us now analyze the case when $n>10t$. Set $n_2=\lfloor\frac{n}{2}\rfloor$ and consider the decomposition
	$$S^{-t+i\beta}_nf = \sum_{k=0}^{n_2+t}A_{n-k}^{i\beta}S_k^{-t-1}f+\sum_{k=n_2+t+1}^nA_{n-k}^{i\beta}S_k^{-t-1}f=:F_{n,t}^1f+F_{n,t}^2f.$$
One can deduce from the triangle inequality and Proposition \ref{JunXu} that 
	\begin{align*}
		\|{\underset{0\leq n\leq \frac{md}{m+1}}{\mathrm{sup}^+}}M_n^{-t+i\beta}f\|_{2} &\leq \Big\|{\underset{0\leq n\leq \frac{md}{m+1}}{\mathrm{sup}^+}}\frac{F_{n,t}^1f}{(1+n)^{-t+1}}\Big\|_2+\Big\|{\underset{0\leq n\leq \frac{md}{m+1}}{\mathrm{sup}^+}}\frac{F_{n,t}^2f}{(1+n)^{-t+1}}\Big\|_2.
	\end{align*}	

We now estimate each term separately.  For the second term, we use Proposition \ref{JunXu} (iii) to obtain
	$$\Big\|{\underset{0\leq n\leq \frac{md}{m+1}}{\mathrm{sup}^+}}\frac{F_{n,t}^2f}{(1+n)^{-t+1}}\Big\|_2 \leq \sup_{0\leq k\leq n}|A_{n-k}^{i\beta}|\sum_{k=n_2+t+1}^n\frac{(1+n)^{t-1}}{(1+k)^{t}}\|{\underset{0\leq n\leq \frac{md}{m+1}}{\mathrm{sup}^+}}M_k^{-t-1}f\|_{2}.$$
The summation can be bounded as follows
	$$\sum_{k=n_2+t+1}^n\frac{(1+n)^{t-1}}{(1+k)^{t}}\leq\frac{(n-n_2-t)(1+n)^{t-1}}{(n_2+t+2)^t}\leq\Big(\frac{n+1}{n_2+2}\Big)^{t-1}\leq2^t.$$
Combining this with Lemma \ref{NS1} (ii) and Lemma \ref{NS4} yields that
	$$\Big\|{\underset{0\leq n\leq \frac{md}{m+1}}{\mathrm{sup}^+}}\frac{F_{n,t}^2f}{(1+n)^{-t+1}}\Big\|_2\leq c_te^{2\beta^2}\|f\|_2.$$

It remains to estimate the first term. Observe that for every nonnegative integer $\ell$ (see \cite[Lemma 2.4]{K1}),
\begin{align}\label{noise10}
\sum_{k=0}^mA_{n-k}^{-\ell+i\beta}S_k^{-t-1+\ell}f =A_{n-m}^{-\ell+i\beta}S_m^{-t-1+\ell}f+\sum_{k=0}^{m-1}A_{n-k}^{-\ell-1+i\beta}S_k^{-t+\ell}f.
\end{align}	
	Then by repeated applications of identity (\ref{noise10}), we obtain
\begin{align*}
	F_{n,t}^1f &= A^{i\beta}_{n-(n_2+t)}S_{n_2+t}^{-t-1}f+\sum_{k=0}^{n_2+t-1}A^{-1+i\beta}_{n-k}S_{k}^{-t}f\\
	&=A^{i\beta}_{n-(n_2+t)}S_{n_2+t}^{-t-1}f+\sum_{\ell=1}^{t-1} A^{-\ell+i\beta}_{n-(n_2+t-\ell)}S_{n_2+t-\ell}^{-t-1+\ell}f+\sum_{k=0}^{n_2}A^{-t+i\beta}_{n-k}S_{k}^{-1}f.
\end{align*}	
Therefore, by Minkowski's inequality and Proposition \ref{JunXu} (iii), we get
	$$\Big\|{\underset{0\leq n\leq \frac{md}{m+1}}{\mathrm{sup}^+}}\frac{F_{n,t}^1f}{(1+n)^{-t+1}}\Big\|_2\leq I\times II,$$	
where		
$$I=\sum_{\ell=0}^{t-1}\|{\underset{0\leq n\leq \frac{md}{m+1}}{\mathrm{sup}^+}}M_n^{-\ell}f\|_2$$		
and
$$II=(1+n)^{t-1}\max\Big\{\sum_{\ell=0}^{t-1}\frac{|A^{-\ell+i\beta}_{n-(n_2+t-\ell)}|}{(n_2+t-\ell+1)^{t-\ell}},\sum_{k=0}^{n_2}|A^{-t+i\beta}_{n-k}|\Big\}.$$
By Lemma \ref{NS4}, $I\leq C_t\|f\|_2$. On the other hand, for each $0\leq\ell\leq t-1$, we use Lemma \ref{NS1} to get
	\[ \aligned
\frac{	(1+n)^{t-1}|A^{-\ell+i\beta}_{n-(n_2+t-\ell)}|}{(n_2+t-\ell+1)^{t-\ell}} &\leq \frac{	(1+n)^{t-1}a_0 e^{3 \beta^2}(n-(n_2+t-\ell)+1)^{-\ell}}{(n_2+t-\ell+1)^{t-\ell}} \\
	&\leq c_{t,\ell} e^{3\beta^2}; \endaligned \]
and the sum
	\[ \aligned
(1+n)^{t-1}|\sum_{k=0}^{n_2} A_{n-k}^{-t + i\beta} |  &\leq (1+n)^{t-1}(n_2 + 1) \cdot \max_{0 \leq k \leq n_2} (n-k+1)^{-t} \cdot B_t e^{3\beta^2} \\
&\leq c_t e^{3\beta^2}. \endaligned \]
Combining these two estimates, we have $II\leq c_t e^{3\beta^2},$ which finishes the proof.	
\end{proof}
Now we are in a position to prove Theorem \ref{generalize1}.  Our approach is based on Stein's complex interpolation. For the sake
of self-completeness, we present a detailed proof.  
\begin{proof}[Proof of Theorem \ref{generalize1}.]
 Write $\alpha=\beta+i\gamma$ with $\beta,\gamma\in\R$. Choose $\theta\in(0,1),\; q\in(1,\infty)$, $a\in\Z$
	and $b>0$ such that
	$$\frac{1}{p}=\frac{1-\theta}{2} + \frac{\theta}{q}
	\quad\mbox{and}\quad
	\beta=(1-\theta)a +\theta\,b.$$
	Let $f\in L_p(\mathcal{N})$ and $g=(g_n)$ be a finite sequence in
	$L_{p'}(\mathcal{N})$ such that $\|f\|_p<1$ and $\|g\|_{L_{p'}(\mathcal{N}; \ell_1)}<1$.
	Define
	$$h(z)=u\,|f|^{\frac{p(1-z)}{2}+ \frac{pz}{q}}\, ,
	\quad z\in\C,$$
	where $f=u|f|$ is the polar decomposition of $f$. By construction, for any $t\in \mathbb{R}$, $\|h(it)\|_{2}=\|f\|_p^{\frac{p}{2}}<1$ and $\|h(1+it)\|_{q}=\|f\|_p^{\frac{p}{q}}<1$. By Proposition \ref{JunXu} (iv), there is a function
	$w=(w_n)_n$ continuous on the strip $\{z\in\C: 0\le{\rm
		Re}(z)\le 1\}$ and analytic in the interior such that $w(\theta)=g$
	and
	$$\sup_{t\in\R}\,\max\Big\{\|w(i\,t)\|_{L_{2}(\mathcal{N}; \ell_1)}\ ,\
	\|w(1+i\,t)\|_{L_{q'}(\mathcal{N}; \ell_1)}\Big\}<1.$$
	Now define
	$$F(z)=\exp\big(\delta(z^2-\theta^2))\,\sum_n\varphi\big[M_n^{(1-z)a + zb+i\,\gamma}(h(z))\,
	w_n(z)\big],$$
	where $\delta>0$ is a constant to be specified later. Clearly, $F$ is a function
	analytic in the open strip $\{z\in\C:0<{\rm Re}(z)< 1\}$ (note that $F$ is a finite sum).
	Applying Lemma \ref{NS41} when $a\ge0$ and Lemma \ref{NS14} when $a\le-1$, we have
	\begin{eqnarray*}|F(it)|
		&\le& \exp\big(\delta(-t^2-\theta^2))\,\big\|\big(M_n^{a +i(-ta +t
			b+\gamma)}(h(it))\big)_n\big\|_{L_{2}(\mathcal{N}; \ell_\infty)}
		\,\big\|w(i\,t)\big\|_{L_{2}(\mathcal{N}; \ell_1)}\\
		&\le& C_\alpha\, \exp\big((-\delta+c_{m,\beta, b,\gamma})t^2-\delta\theta^2)\,\|h(it)\|_2
		\\ &\le& C_\alpha\, \exp\big((-\delta+c_{m,\beta, b,\gamma})t^2-\delta\theta^2)\ .
	\end{eqnarray*}
	Similarly, by Lemma \ref{NS41},
	$$|F(1+it)|\le C_{\alpha, q}\, \exp\big((-\delta+c'_{m,\beta, b,\gamma})t^2+\delta(1-\theta^2))\ .$$
	Choosing $\delta>\max(c_{m,\beta, b,\gamma},\; c'_{m,\beta, b,\gamma})$,
	we get
	$$\sup_{t\in\R}\,\max\Big\{|F(i\,t)|\ ,\
	|F(1+i\,t)|\Big\}\le C_{m,\alpha, p}$$
for some constant $C_{m,\alpha, p}>0$.
	Then by the maximum principle,
	$|F(\theta)|\le  C_{m,p,\beta,b,\gamma},$
	which by duality implies
	$$\big|\sum_n\varphi\big[M_n^{\alpha}(f)\,
	w_n\big]\big|\le C_{m,\alpha, p}\ ,$$
as $g=(g_n)$ is an abitrary finite sequence such that $\|g\|_{L_{p'}(\mathcal{N}; \ell_1)}<1$. This proves the first inequality.
	 The second conclusion follows similarly using the estimates involving $N_n^\alpha(f)$ from the previous lemmas. The details are left to readers.
\end{proof}

\section{Proof of Theorem \ref{main1}}
In this section, we prove Theorem \ref{main1}.
\begin{proof}[Proof of Theorem \ref{main1}.] Let $(\M,\tau)$ be a semifinite von Neumann algebra and $\alpha: \Z_{m+1}^d\to \text{Aut}(\mathcal{M},\tau)$ be a $\tau$-preserving action. Then for every $s\in \Z^d_{m+1}$, the map $\alpha_{s}$ extends to a positive isometry on $L_{p}(\M)$ for all $p\geq 1$. It follows from \cite[Proposition 5.3]{HRW} that $(\alpha_{s}\otimes id_{\ell_{\infty}})_{s\in\Z^d_{m+1}}$ extends to a family of isometry on $L_{p}(\M;\ell_{\infty})$ for $1\leq p\le \infty$, still denoted by $(\alpha_{s})_s$. Hence,
for each $s\in \Z^d_{m+1}$, 
$$\big\|(M_kx)_{1\leq k\leq d}\big\|_{L_{p}(\M;\ell_{\infty})}=\big\|(\alpha_{-s}M_kx)_{1\leq k\leq d}\big\|_{L_{p}(\M;\ell_{\infty})},$$
which implies that
\begin{equation}\label{eq:alpha_g An p p}
	\big\|(M_kx)_{1\leq k\leq d}\big\|^p_{L_{p}(\M;\ell_{\infty})}=
	\frac{1}{|\Z^d_{m+1}|}
	\sum_{s\in \Z^d_{m+1}} \big\|(\alpha_{-s}M_kx)_{1\leq k\leq d}\big\|_{L_{p}(\M;\ell_{\infty})}^{p}.
\end{equation}
Given $x$, define an operator-valued function $f\in L_{p}(\Z^d_{m+1}; L_p(\mathcal{M}))$ as
\begin{equation*}
	f(s)=\alpha_{-s}x,\quad s\in \Z^d_{m+1}.
\end{equation*}
Then clearly for each $s\in \Z^d_{m+1}$,
\begin{align}\label{oper}
	\alpha _{-s}M_kx=
\frac{1}{\binom{d}{k}m^k}\sum_{u:|u|=k}\alpha_{u-s}x=\frac{1}{\binom{d}{k}m^k}\sum_{u:|u|=k}f(s-u)=
	T_kf(s).
\end{align}
By the definition of $\bigl\| (T_kf)_{1\leq k\leq d}
	\bigr\| _{L_{p}(L_{\infty }(\Z_{m+1}^{d})
		\overline{\otimes}\M;\ell _{\infty })}$,  there exists a factorization $T_kf=ay_kb$ with
$a,b\in L_{2p}(L_{\infty}(\Z_{m+1}^{d})\overline{\otimes}\M)$ and
$y_k\in L_{\infty }(\Z_{m+1}^{d})\overline{\otimes}\M$ such that
\begin{equation*}
	\| a \|_{2p}
	\sup _{1\leq k\leq d} \| y_k \| _{\infty }
	\| b \| _{2p}\leq \bigl\| (T_k{f})_{1\leq k\leq d}
	\bigr\| _{L_{p}(L_{\infty }(\Z_{m+1}^{d})
		\overline{\otimes}\M;\ell _{\infty })}+\delta .
\end{equation*}
By  H\"{o}lder's inequality, one has
\begin{align*}
&\sum_{s\in\Z_{m+1}^d} \big\|(T_k{f}(s))_{1\leq k\leq d}
	\big\| _{L_{p}(\M;\ell _{\infty })}^{p}  \\ \leq &\sum_{s\in\Z_{m+1}^d} \big\| a(s) \big\| _{2p}^{p}
	\sup _{1\leq k\leq d} \big\| y_k(s) \big\|
	_{
		\infty }^{p} \big\| b(s) \big\|
	_{2p}^{p}
	\\
\leq &\big(\sum_{s\in\Z_{m+1}^d} \big\| a(s) \big\| _{2p}^{2p}\big)^{\frac{1}{2}}
	\sup _{1\leq k\leq d}\sup_{s\in \Z_{m+1}^d} \big\| y_k(s) \big\|
	_{
		\infty }^{p}
\big(\sum_{s\in\Z_{m+1}^d}\big\| b(s) \big\|
	_{2p}^{2p}\big)^{\frac{1}{2}}
	\\
	\leq &\| a \|_{2p}^{p}
	\sup _{1\leq k\leq d} \| y_k \| _{\infty }^{p}
	\|b \| _{2p}^{p}
	\\
	\leq &\big(\| (T_k{f})_{1\leq k\leq d}
	\| _{L_{p}(L_{\infty }(\Z_{m+1}^{d})
	\overline{\otimes}\M;\ell _{\infty })}+\delta
	\big)^{p}.
\end{align*}
By the arbitrariness of $\delta$, we obtain
\begin{align}\label{oper1}
	\sum_{s\in\Z_{m+1}^d} \|(T_k{f}(s))_{1\leq k\leq d}
	\| _{L_{p}(\M;\ell _{\infty })}^{p} \leq \| (T_k{f})_{1\leq k\leq d}
	\| _{L_{p}(L_{\infty }(\Z_{m+1}^{d})
	\overline{\otimes}\M;\ell _{\infty })}^{p}.
\end{align}
Now
using (\ref{eq:alpha_g An p p}), (\ref{oper}) and (\ref{oper1}), we deduce that
\begin{align*}
	\big\|(M_kx)_{1\leq k\leq d}\big\|^p_{L_{p}(\M;\ell_{\infty})} &=
		\frac{1}{|\Z^d_{m+1}|}
	\sum_{s\in \Z^d_{m+1}} \big\|(T_kf(s))_{1\leq k\leq d}\big\|_{L_{p}(\M;\ell_{\infty})}^{p}\\&\leq
	\frac{1}{|\Z^d_{m+1}|}\| (T_kf)_{1\leq k\leq d}
	\| _{L_{p}(L_{\infty }(\Z_{m+1}^{d})
		\overline{\otimes}\M;\ell _{\infty })}^{p}
	\\
	&\le C_{m,p}
	\frac{1}{|\Z^d_{m+1}|} \|f\| _{L_{p}(\Z_{m+1}^{d};L_{p}(\M))}^{p}\\&=C_{m,p}
	\frac{1}{|\Z^d_{m+1}|}\sum_{s\in \Z_{m+1}^{d}}
 \| \alpha _{s}x \|
	_{L_{p}(\M)}^{p}
\\&=C_{m,p}\| x \|
	_{L_{p}(\M)}^{p},
\end{align*}
where the second inequality above uses the strong type $(p,p)$ estimate of $(T_k)_{1\leq k\leq d}$ by Theorem \ref{main2}, and the last equality follows from the fact that $\alpha _{s}$ is isometric on $L_{p}(\M)$ for each $s\in\Z_{m+1}^d$.
This proves the first estimate. 

For the second assertion, let $x\in L_p(\M)$.  Note that for each $k$, $M_k$ is unital and completely positive. Then by Kadison-Schwarz inequality for unital $2$-positive map, we have
$M_k(x^*x)\leq M_k(x)^*M_k(x).$
Now applying the first assertion to $x^*x$, we deduce from Remark \ref{rk:MaxFunct} that there exists positive $b\in L_{p/2}(\M)$ such that
$$\|b\|_{p/2}\leq C_{p,m}\|x^*x\|_{p/2}=C_{p,m}\|x\|_{p}^2\ \mbox{and}\ M_k(x^*x)\leq b,\ \forall k.$$
Hence, $M_k(x)^*M_k(x)\le M_k(x^*x)\leq b$, which implies that for each $k$, there is a $y_k\in\M$ such that $M_k(x)=y_kb^{1/2}$ for all $k$. This gives the desired factorization of $(M_k(x))_k$ as
an element in $L_p(\M;\ell^c_\infty)$ and the proof is finished.
\end{proof}


\begin{remark}{\rm The above transference method actually shows that Theorem \ref{main1} for the fully noncommutative setting and Theorem \ref{main2} for the operator-valued setting are equivalent with identical constants.
}
\end{remark}

We now extend the main result-Theorem \ref{main1} for state preserving action on general von Neumann algebras. This is done by the standard Haagerup reduction method. The readers are referred to \cite{HJX,GYZ,P2,Terp} for more information on this topic. For completeness, we also provide a brief appendix on the basic properties we will use.

 Let $\mathcal{M}$ be a von Neumann algebra equipped with a normal faithful state $\phi$. We denote by $\mathcal{L}_p(\M,\phi)$ the Haagerup $L_p$-space associated to $\phi$ and
$\mathcal{L}_p(\M,\phi;\ell_\infty)$ the $\ell_\infty$-valued $L_p$-spaces. See \cite[Section 7]{HJX} and our appendix.
We say $\alpha: \Z_{m+1}^d\to \text{Aut}(\M)$ is $\phi$-preserving action on $\M$ if i) for each $s$,  $\alpha_s$ is $\phi$-preserving (i.e. $\phi\circ \alpha_s=\phi$) $*$-automorphism on $\M$, ii) for any $s,t\in \Z_{m+1}^d$, $\alpha_s\circ\alpha_t=\alpha_{s+t}$. The main theorem for general von Neumann algebra is stated as follows.
\begin{thm}\label{main3}
Let $\mathcal{M}$ be a von Neumann algebra equipped with a normal faithful state $\phi$.
 Let $\alpha: \Z_{m+1}^d\to \text{Aut}(\M)$ be a $\phi$-preserving action such that for all $s\in \Z_{m+1}^d$:
  $\alpha_s$ commutes with the modular automorphism group $\sigma_{t}^{\phi}$, i.e. $\alpha_s\circ\sigma_{t}^{\phi}=\sigma_{t}^{\phi}\circ \alpha_s$.  
 Then for $p>1$, the spherical mean $M_k(\alpha)$ defined in \eqref{discrete11} satisfies
	$$\|(M_k(\alpha)x)_{1\leq k\leq d}\|_{\mathcal{L}_p(\M,\phi;\ell_\infty)}\leq C_{p,m}\|x\|_{\mathcal{L}_p(\M,\phi)},\ \forall x\in \mathcal{L}_p(\M,\phi);$$
	and for $p>2$,
	$$\|(M_k(\alpha)x)_{1\leq k\leq d}\|_{\mathcal{L}_p(\M,\phi;\ell^c_\infty)}\leq \sqrt{C_{p,m}}\|x\|_{\mathcal{L}_p(\M,\phi)},$$
	where the constant $C_{p,m}$ depends only on $p$ and $m$.
\end{thm}
\begin{proof}[Reduction to the tracial case]The argument is almost identical to \cite[Theorem 7.9]{HJX}. Let $\widehat{M}$ be the crossed product algebra and $(\widehat{M}_n,\tau_n)$ be the sequence of finite tracial subalgebra from the Haagerup reduction.
Here we note that the assumption on commutation with modular group $\sigma^\phi$ is natural, as by \cite[Theorem 4.1]{HJX}, it guarantees an extension  $\widehat{\alpha}: \Z_{m+1}^d\to \text{Aut}(\widehat{\M})$ satisfying the following properties:
\begin{itemize}
\item[i)] $\widehat{\alpha}$ is a $\widehat{\phi}$-preserving action that commutes the modular group  $\sigma^{\hat{\phi}}$:
\[ \widehat{\alpha}_s\circ\sigma^{\hat{\phi}}_t=\sigma^{\hat{\phi}}\circ \widehat{\alpha}_s.\ \]
\item[ii)] For each $n\ge 1$, $\widehat{\alpha}$ is also invariant on $\lambda(G)$ 
\[\widehat{\alpha}_s(\lambda(g))=\lambda(g). \]
This implies that  $\widehat{\alpha}$ commutes with conditional expectation $E_n$,
\begin{align}\widehat{\alpha}_s\circ E_n=E_n\circ \widehat{\alpha}_s .\label{eq:commute}\end{align}
Moreover, the restriction $\widehat{\alpha}|_{\widehat{M}_n}$ is a $\widehat{\phi}$-preserving action
on the subalgebra that also preserves the trace $\tau_n=\frac{\phi_n}{\phi_{n}(1)}$
\[ \phi_n(\cdot )=\widehat{\phi}(\cdot \ e^{a_n} )\ ,\  a_n=i2^n\text{Log}(\lambda(2^{-n})) .\]
\end{itemize}
Now we consider $x$ as an element in $\mathcal{L}_p(\widehat{\M},\widehat{\phi})$ and then apply the conditional expectation $E_n$ to it: $x_n:=E_n(x)\in \mathcal{L}_p(\widehat{\M}_n,\widehat{\phi})$. We can use the tracial case  Theorem \ref{main1} to $\widehat{\alpha}$ on $(\widehat{\M}_n,\tau_n)$ and obtain that for each $n\ge 1$, 
	$$\|(M_k(\widehat{\alpha})x_n)_{1\leq k\leq d}\|_{\mathcal{L}_p(\widehat{\M}_n,\tau_n;\ell_\infty)}\leq C_{p,m}\|x\|_{\mathcal{L}_p(\widehat{\M}_n,\tau_n)}\ \forall n\in \N,$$
	where for $1\leq k\leq d$, $M_k(\widehat{\alpha})$ is the spherecial mean for the extension action $\widehat{\alpha}$
	$$M_k(\widehat{\alpha})=\frac{1}{\binom{d}{k}m^k}\sum_{u:|u|=k}\widehat{\alpha}_u.$$
By the isometry of (maximal) Haagerup $L_p$ space for different states,  we also have
$$\|(M_k(\widehat{\alpha})x_n)_{1\leq k\leq d}\|_{\mathcal{L}_p(\widehat{\M}_n,\widehat{\phi};\ell_\infty)}\leq C_{p,m}\|x\|_{\mathcal{L}_p(\widehat{\M}_n,\widehat{\phi})}. $$
By the martingale convergence theorem (see e.g. \cite[Remark 6.1]{HJX}), one has
	$$\lim_{n\rightarrow\infty}x_n=x\ \ \text{in}\  \mathcal{L}_p(\widehat{\mathcal{M}}, \hat{\phi}).$$
Also, by the commutation relation \eqref{eq:commute}
	$$\widehat{\alpha}_ux_n=\widehat{\alpha}_uE_n(x)=E_n(\widehat{\alpha}_ux).$$
	It then follows that
	$$\lim_{n\rightarrow\infty}\|(M_k(\widehat{\alpha})x_n)_{1\leq k\leq d}\|_{\mathcal{L}_p(\M,\phi;\ell_\infty)}=\|(M_k(\widehat{\alpha})x)_{1\leq k\leq d}\|_{\mathcal{L}_p(\M,\phi;\ell_\infty)}.$$
	Since the constant $C_{p,m}$ is independent of $n$ and $\widehat{\alpha}$ is an extension of $\alpha$, we have for $x\in \mathcal{L}_p(\widehat{\M},\widehat{\phi})$, 
$$\|(M_k(\alpha)x)_{1\leq k\leq d}\|_{\mathcal{L}_p(\M,\phi;\ell_\infty)}\leq C_{p,m}\|x\|_{\mathcal{L}_p(\M,\phi)}.$$
This proves the first part. The second part can be done similarly.
\end{proof}

\section{Examples}
In this section, we give some examples of Theorem \ref{main1} and Theorem \ref{main3}.
\subsection{Quantum Boolean cubes}  The quantum analog of  the classical Boolean cube $\{-1,1\}^n$ is the matrix algebra $M_2(\C)^{\otimes n}\cong M_{2^n}(\C)$. We denote by $\mbox{tr}$ the standard matrix trace on $M_2(\C)^{\otimes n}$ and by $\tau=\frac 1{2^n}\,\mathrm{tr}(\cdot)$ the normalized trace. The normalized Schatten-$p$ norm of $A\in M_2(\C)^{\otimes n}$ is defined by 
\begin{align*}
	\|A\|_p=\tau(|A|^p)^{\frac{1}{p}}
\end{align*}
for $1\leq p<\infty$ and $\|\cdot\|_{\infty}\equiv \|\cdot\|$ is the usual operator norm. Recall that Pauli matrices
\begin{equation*}
	\sigma_0=\begin{pmatrix}1&0\\0&1\end{pmatrix},\quad \sigma_1=\begin{pmatrix}0&1\\1&0\end{pmatrix},\quad \sigma_2=\begin{pmatrix}0&-i\\i&0\end{pmatrix},\quad \sigma_3=\begin{pmatrix}1&0\\0&-1\end{pmatrix}
\end{equation*}
form an orthonormal basis of $M_2(\C)$ w.r.t. the normalized trace inner product. By the relation $\sigma_1\sigma_2=-i\sigma_3$, this introduces an action of $\alpha:\mathbb{Z}_2^3 \curvearrowright M_2(\C)$ generated by the unitary conjugation action that
\begin{align*} &\alpha_{000}(A)= \alpha_{111}(A) =A\ , \  \alpha_{001}(A)=\alpha_{110}(A)=\sigma_1A\sigma_1\ ,\\ \ &\alpha_{010}(A)=\alpha_{101}(A)=\sigma_2A\sigma_2\ ,\ \alpha_{100}(A)=\alpha_{011}(A)=\sigma_3A\sigma_3.\ 
\end{align*}
This induces a map $\phi:\mathbb{Z}_2^3\to  \{0,1,2,3\}$:
\[ 000,111\mapsto 0\ ,\  001,110\mapsto 1\ ,\  010,101\mapsto 2\ ,\  100,011\mapsto 3. \]
For a multi-index ${\bf s}=(s_1,\dots, s_n)\in\{0,1,2,3\}^n$, the product Pauli operators are
\begin{equation*}
	\sigma_{\bf s}:=\sigma_{s_1}\otimes\dots\otimes \sigma_{s_n}.
\end{equation*}
We define the product action $\alpha:\mathbb{Z}_2^{3n} \curvearrowright M_2(\C)^{\otimes n}$ that for ${\bf x}=(x_1,\cdots, x_{3n})\in \mathbb{Z}_2^{3n}$
\[ \alpha_{{\bf x}}(A)=\sigma_{\phi({\bf x})}A\sigma_{\phi({\bf x})}, \]
where with slight abuse of notation, 
\[\phi({\bf x})=\phi(x_1,\cdots,x_{3n})=(\phi(x_1 x_2 x_3), \phi(x_4x_5x_6), \cdots,\phi(x_{3n-2} x_{3n-1}x_{3n}))\in  \{0,1,2,3\}^n.\]
Now consider the spherical mean 
$$M_kA=\frac{1}{\binom{3n}{k}}\sum_{{\bf x}\in\mathbb{Z}_2^{3n}, |{\bf x}|=k}\alpha_{{\bf x}} A,\ A\in M_2(\C)^{\otimes n}.$$
We decompose the summation of $|{\bf x}|=k$ into different partitions. For ${\bf x}=(x_1,\cdots, x_{3n})\in \mathbb{Z}_2^{3n} $ and ${\bf k}=(k_1,\cdots, k_n)\in \{0,1,2,3\}^n$, we write  
\[{\bf x}\prec {\bf k}\ \  \text{ if }\ \  x_{3i+1}+x_{3i+2}+x_{3i+3}=k_i\  \text{ for each } i.\]
Then 
\begin{align*} 
M_kA=&\frac{1}{\binom{3n}{k}}\sum_{{\bf x}\in\mathbb{Z}_2^{3n}, |{\bf x}|=k}\alpha_{{\bf x}} A\\ 
=&\frac{1}{\binom{3n}{k}}\sum_{{\bf k}\in\{0,1,2,3\}^n, |{\bf k}|=k}\left( \sum_{{\bf x}\prec \bf k}\alpha_{{\bf x}} A\right).\\ 
\end{align*}
Observe that for each $\bf k$, 
\begin{align*}
&\sum_{{\bf x}\prec \bf k}\alpha_{{\bf x}} A\\
=& \sum_{|x_1x_2x_3|=k_1}\sum_{|x_4x_5x_6|=k_2}\cdots \sum_{|x_{3n-2}x_{3n-1}x_{3n}|=k_n} \alpha_{x_1x_2x_3}\otimes \alpha_{x_4x_5x_6}\otimes\cdots  \otimes \alpha_{x_{3n-2}x_{3n-1}x_{3n}}(A)\\
=& \left(\sum_{|x_1x_2x_3|=k_1}\alpha_{x_1x_2x_3} \right)\otimes \left(\sum_{|x_4x_5x_6|=k_2}\right)\cdots \otimes\left(\sum_{|x_{3n-2}x_{3n-1}x_{3n}|=k_n} \alpha_{x_{3n-2}x_{3n-1}x_{3n}}\right)(A).
\end{align*}
Note that 
\[ \sum_{|x_1x_2x_3|=1}\alpha_{x_1x_2x_3}(A)=\sum_{|x_1x_2x_3|=2}\alpha_{x_1x_2x_3}(A)=\sum_{s=1}^3\sigma_{s}A\sigma_s=4\tau(A)I-A. \]
On $M_2(\C)^{\otimes n}$, write $$\tilde{E}_i(A):=4\tau_{i}(A)\ten I_i-A,$$ where $\tau_{i}$ (resp. $I_i$) is the normalized partial trace (resp. identity operator) on $i$-th component, and for ${\bf k}\in\{0,1,2,3\}^n$, the product map
$$\tilde{E}_{{\bf k}}:=\bigotimes_{k_i=1 \text{ or } 2}  \tilde{E}_i.$$ Then we can write the spherical mean as
\begin{align*} 
M_kA 
=\frac{1}{\binom{3n}{k}}\sum_{{\bf k}\in\{0,1,2,3\}^n, |{\bf k}|=k} \tilde{E}_{{\bf k}}(A).
\end{align*}
Thus each $M_k$ gives a linear combination of partial traces (marginals) $\{\tau_I(A), I\subset [n]\}$ of $A$, which in general are not commutative on $M_2(\C)^{\otimes n}$. For example,
\begin{align*} 
M_1A=M_{3n-2}A 
=\frac{4}{3}\Big(\frac{1}{n}\sum_{i=1}^n \tau_{i}(A)\ten I_i\Big)- \frac{1}{3}A.
\end{align*}
However,
\begin{align*} 
M_2A=M_{3n-1}A &\\
=\frac{2}{3n(3n-1)}&\Big(\sum_{1\le i,j \le n, i\neq j}  \tilde{E}_i\otimes \tilde{E}_j(A)+\sum_{i=1}^n  \tilde{E}_i(A)\Big)\\
=\frac{2}{3n(3n-1)}&\Big(\sum_{1\le i,j \le n, i\neq j} \big( 16\tau_{i,j}(A)\otimes I_{i,j}- 4\tau_{i}(A)\otimes I_{i}-4\tau_{j}(A)\otimes I_{j}+A\big)\\ &+\sum_{i=1}^n  4\tau_{i}(A)\ten I_i- A \Big)\\
=\frac{2}{3n(3n-1)}&\Big(16 \sum_{1\le i,j \le n, i\neq j} \tau_{i,j}(A)\otimes I_{i,j}-4(n-2) \sum_{1\le i \le n} \tau_{i}(A)\otimes I_{i}\\ &+ (n^2-n-1)A\Big).
\end{align*}

Therefore, by Theorem \ref{main1} and Remark \ref{rk:MaxFunct}, we deduce that for any $A\in M_2(\C)^{\otimes n}_+$ and every $p>1$, there is $B\in M_2(\C)^{\otimes n}_+$ such that for every $1\leq k\leq 3n$,
$$M_kA 
=\frac{1}{\binom{3n}{k}}\sum_{{\bf k}\in\{0,1,2,3\}^n, |{\bf k}|=k} \tilde{E}_{{\bf k}}(A)\leq B\ \text{and}\ \|B\|_p\leq C_{p}\|A\|_p,$$
where $C_{p}$ is the constant depending only on $p$ (independent of $n$).

\subsection{Generalized Pauli Matrices} Let $m\geq1$ and  $\omega=\omega_{m+1}=\exp(\frac{2\pi i}{m+1})$. Define the $(m+1)$-dimensional shift and clock matrices by
$$\ X|j\rangle=|j+1\rangle\ ,\ Z|j\rangle=\omega^j|j\rangle\ , \ \forall j\in\Z_{m+1}, $$
where the addition $j+1$ is in $\Z_{m+1}$ ($\mod{m+1}$). 
Then $Z^{m+1}=X^{m+1}=I$ and $XZ=\omega^{-1} ZX$.
The generalized Pauli matrices (also called Weyl operators) are defined by
\[
W(a,b)=\omega^{-ab/2}X^{a}Z^{b},\qquad a,b\in\Z_{m+1}.
\]
This induces an action $\alpha: \Z_{m+1}^2 \curvearrowright  M_{m+1}(\C)$ defined as 
\[ \alpha_{ab}(A)=W(a,b)A W(a,b)^\dagger\  , \ A\in  M_{m+1}(\C).\]
Recall the conditional expectations
\begin{align*}&E_{X}(A)= \frac{1}{m+1}\sum_{a=0}^m W(a,0)A W(a,0)^\dagger,\\  &E_{Z}(A)=  \frac{1}{m+1}\sum_{b=0}^m W(0,b)A W(0,b)^\dagger,\\  &E_{\C}(A)=\tau(A)I  \end{align*}
are the projections from $M_{m+1}(\C)$ onto the commutative subalgebra generated by $X$, $Z$ and the scalar matrices, respectively. On $M_{m+1}(\C)$, we have the mean operator $k=1$
\begin{align*} M_1(A)=&\frac{1}{2m}(\sum_{a=1}^m \alpha_{a0}A+\sum_{b=1}^m \alpha_{0b}A )=\frac{m+1}{2m}(E_X(A)+E_Z(A))-\frac{1}{m}A\\ &=\frac{1}{2m}(\tilde{E}_X+\tilde{E}_Z)(A),
\end{align*}
where $\tilde{E}_X=\sum_{a}=1^m\alpha_{a0}=(m+1)E_X-\tau(\cdot)I$ and similarly for $\tilde{E}_Z$, and for $k=2$
\begin{align*}
M_2(A)&=\frac{1}{m^2}\sum_{a,b=1}^m \alpha_{ab}A =\tilde{E}_X\tilde{E}_Z(A).
\end{align*}

Now for $\mathbf{a}=(a_1,\dots ,a_n)$ and $\mathbf{b}=(b_1,\dots ,b_n)\in\Z_{m+1}^{n}$, define
\[
W(\mathbf{a},\mathbf{b})=\bigotimes_{j=1}^{n}W(a_j,b_j).
\]
Consider the product action $\alpha: \Z_{m+1}^{2n} \curvearrowright M_{m+1}(\C)^{\otimes n}$,
\[ \alpha_{\mathbf{a}\mathbf{b}}(A)=W(\mathbf{a},\mathbf{b})A W(\mathbf{a},\mathbf{b})^{\dagger}.\]
Following similar calculations from the previous subsection, we have
\begin{align*} 
M_k^{\alpha} A=&\frac{1}{\binom{2n}{k}}\sum_{ ({\bf a}{\bf b})\in\mathbb{Z}_{m+1}^{2n}, |({\bf a}{\bf b})|=k}\alpha_{{\bf a}{\bf b}} A\\ 
=&\frac{1}{\binom{2n}{k}}\sum_{{\bf k}\in\{0,1,2\}^n, |{\bf k}|=k} \Big(\bigotimes_{i \text{ s.t. }k_i=1}(\tilde{E}_{X_i}+\tilde{E}_{Z_i} )\otimes \bigotimes_{i \text{ s.t. } k_i=2}\tilde{E}_{X_i}\tilde{E}_{Z_i}   \Big)(A),
\end{align*}
where the map $\tilde{E}_{X_i}$ (resp. $\tilde{E}_{Z_i}$) are copies of $\tilde{E}_{X}$ (resp. $\tilde{E}_{Z}$) on the $i$-th component of $M_{m+1}(\C)^{\otimes n}$. It then follows from Theorem \ref{main1} and Remark \ref{rk:MaxFunct} that for any $A\in M_{m+1}(\C)^{\otimes n}_+$ and every $p>1$, there is $B\in M_{m+1}(\C)^{\otimes n}_+$ such that for every $1\leq k\leq 2n$,
$$M_k^{\alpha} A\leq B\ \text{and}\ \|B\|_p\leq C_{p,m}\|A\|_p,$$
where $C_{p,m}$ is the constant depending only on $p$ and $m$ (independent of $n$).

Alternatively, we can consider the sub-action only by $Z$ operators
\[\beta:  \Z_{m+1}^{n} \curvearrowright  M_{m+1}(\C)^{\otimes n}\ ,\  \beta_{{\bf b}}(A)=\alpha_{{\bf 0b}}(A).\]
In this case, the spherical mean is 
\begin{align*} 
M_k^\beta A=&\frac{1}{\binom{n}{k}}\sum_{ {\bf b}\in\mathbb{Z}_{m+1}^{n}, |{\bf b})|=k}\beta_{{\bf b}} A\\ 
=&\frac{1}{\binom{n}{k}}\sum_{J\subset [n]} \Big(\bigotimes_{j \in I }\tilde{E}_{Z_j} \Big)(A).
\end{align*}
Therefore, by Theorem \ref{main1} and Remark \ref{rk:MaxFunct}, for any $A\in M_{m+1}(\C)^{\otimes n}_+$ and every $p>1$, there is $B\in M_{m+1}(\C)^{\otimes n}_+$ such that for every $1\leq k\leq 2n$,
$$M_k^\beta A=\frac{1}{\binom{n}{k}}\sum_{J\subset [n]} \Big(\bigotimes_{j \in I }\tilde{E}_{Z_j} \Big)(A)\leq B\ \text{and}\ \|B\|_p\leq C_{p,m}\|A\|_p.$$

\subsection{Quantum Tori} Let $d\geq2$ and $\theta=(\theta_{kj})$ be a real skew symmetric $d\times d$-matrix.  The associated $d$-dimensional noncommutative torus, denoted by $\mathcal{A}_\theta^d$,  is the universal C$^*$-algebra generated by
$d$ unitary operators $U_1,\cdot\cdot\cdot,U_d$ satisfying
$$U_kU_j=e^{2\pi i\theta_{kj}}U_jU_k,\ j,k=1,\cdot\cdot\cdot,d.$$
Let $U=(U_1,\cdot\cdot\cdot,U_d)$. For $n=(n_1,\cdot\cdot\cdot,n_d)$, write
$$U^n=U_1^{n_1}\cdot\cdot\cdot U_d^{n_d}.$$
The $*$-algebra $\mathcal{P}_\theta$ consists of all finite linear combinations
$$x=\sum_{n\in\Z^d}a_nU^n,\ \text{with}\ a_n\in\C.$$
Recall that $\mathcal{P}_\theta$
is dense in $\mathcal{A}_\theta^d$. The linear functional $x\mapsto a_0$ on $\mathcal{P}_\theta$ extends to a faithful tracial state $\tau$ on $\mathcal{A}_\theta^d$.
The $d$-dimensional quantum tori $\mathbb{T}_\theta^d$ is defined as the w$^*$-closure of $\mathcal{A}_\theta^d$ in the  GNS representation
of $\tau$, where $\tau$ extends to a normal faithful tracial state on $\mathbb{T}_\theta^d$, still denoted by $\tau$. We write $L_p(\mathbb{T}_\theta^d)$ for the noncommutative $L_p$-spaces associated to the pair $(\mathbb{T}_\theta^d,\tau)$.

Let $m \ge 1$ and set $\omega = \exp\bigl(\frac{2\pi i}{m+1}\bigr)$.  We consider the action $\alpha : \mathbb Z_{m+1}^d \curvearrowright \mathbb{T}_\theta^d$ defined on the generators by
\[
\alpha_z(U^n) = \omega^{z\cdot n }\,U^n,\ z = (z_1,\dots ,z_d) \in \mathbb Z_{m+1}^d,
\]
where $z\cdot n=\sum_{i=1}^d z_in_i$.
For $1 \le k \le d$ define the spherical mean
\[
M_k x = \frac{1}{\binom{d}{k}\,m^{\,k}}
\sum_{\substack{z \in \mathbb Z_{m+1}^d \\ |z| = k}} \alpha_z(x), \ x \in \mathbb{T}_\theta^d,
\]
where $|z|$ denotes the number of coordinates of $z$ different from $0$.
A direct computation on the monomials $U^n$ gives
\begin{align*}
	M_k U^n
	&= \frac{1}{\binom{d}{k}}
	\sum_{\substack{J\subseteq[d]\\|J|=k}}
	\Bigl( \frac{1}{m^k} \sum_{\substack{z:\operatorname{supp}z=J}} \alpha_z(U^n) \Bigr) \\[2mm]
	&= \frac{1}{\binom{d}{k}}
	\sum_{\substack{J\subseteq[d]\\|J|=k}}
	 \prod_{i\in J} \Bigl(\frac{1}{m}\sum_{\ell=1}^{m} \omega^{\ell n_i}
	 \Bigr) U^n \\[2mm]
	&= \frac{1}{\binom{d}{k}}
	\sum_{\substack{J\subseteq[d]\\|J|=k}}
	 \prod_{i\in J} \beta(n_i)  U^n:=\frac{1}{\binom{d}{k}}
	\sum_{\substack{J\subseteq[d]\\|J|=k}}
	\beta_{J}(n)  U^n,
\end{align*}
where
\begin{align*}
&\beta(n_i) = \frac{1}{m}\sum_{\ell=1}^{m} \omega^{\ell n_i}
= \frac{1}{m}\sum_{\ell=1}^{m} e^{2\pi i \ell n_i/(m+1)}=\begin{cases}
1 & \text{if } m+1 \mid n_i , \\
-\frac{1}{m} & \text{if } m+1 \nmid n_i,
\end{cases}\\
&\beta_J(n)=\prod_{i\in J}\beta(n_i).
\end{align*}
Namely, $M_k$ is a Fourier multiplier 
\[ M_k(\sum_{n\in \mathbb{Z}^d} a_n U^n)= \sum_{n\in \mathbb{Z}^d} a_n \beta_{k}(n)U^n \]
with the multiplier function \[\beta_{k}(n)=\frac{1}{\binom{d}{k}}
	\sum_{\substack{J\subseteq[d]\\|J|=k}}
	\beta_{J}(n).\]
Write $l(n)$ as the number of $n_i$ in $n=(n_1,\cdots, n_d)$ such that $m+1 \nmid n_i$. Then
\[\beta_{k}(n)=\frac{1}{\binom{d}{k}}\sum_{1\le l\le \min\{k, l(n)\}}  \binom{d-l(n)}{k-l}\binom{l(n)}{l} (-\frac{1}{m})^{l}.\]
For example, if $m+1$ is a prime number,  $l(n)$ is the number of $n_i$'s which are multiple of $m+1$. As a consequence, we apply Theorem~\ref{main1} to obtain that for every $p>1$ and any $x\in L_p^{+}(\mathbb{T}_\theta^d)$, there exists a constant $C_{p,m}$ independent of the dimension $d$ and a positive operator  $a\in L_p^{+}(\mathbb{T}_\theta^d)$ such that
\[
M_kx\le a,\quad \forall 1\le k\le d\ \text{and}\ \|a\|_p\leq C_{p,m}\|x\|_p.
\]
\subsection{Hyperfinite III$_\lambda$ factor}

Let $0<\lambda<1$. Consider hyperfinite III$_\lambda$ factor defined as
\[ \mathcal{R}_\lambda:=\bigotimes_{i=1}^\infty (M_2(\C),\varphi_\lambda).\]
Here, for each $n\ge 1$, we have the natural embedding
\[ \bigotimes_{i=1}^d (M_2(\C),\varphi_\lambda)\subset \bigotimes_{i=1}^{n+1} (M_2(\C),\varphi_\lambda), x\mapsto x\otimes I_2 \]
and the finite tensor product $$\mathcal{R}_\lambda=\overline{\bigcup_{n=1}^\infty \bigotimes_{i=1}^n (M_2(\C),\varphi_\lambda)^{w^*}}$$ is interpreted as the $w^*$-closure of all finite tensor product acting on GNS representation of product state $\phi_\lambda:=\otimes_{i=1}^\infty \varphi_\lambda$, which 
$\phi_\lambda$ is a normal faithful state on $\mathcal{R}_\lambda$.   

On $(M_2(\C),\varphi_\lambda)$, we have a $\varphi_\lambda$-preserving $\mathbb{Z}_2$ action by the Pauli $Z$ matrix,
\[\alpha(A)=ZAZ, \   Z=\left[\begin{array}{cc} 1 &0\\ 0& -1
\end{array}\right].\]
Moreover, it commutes with the modular group $\alpha\sigma^{\varphi_\lambda}=\sigma^{\varphi_\lambda}\alpha$. 
For each $i\ge 1$, this induces a $\mathbb{Z}_2$-action on the $i$-th component
\[ \alpha_i(A)=Z_iAZ_i,  A\in \mathcal{R}_\lambda\]
where $Z_i=I\otimes \cdots\otimes Z\otimes  \cdots\otimes I$ is the Pauli $Z$ matrix on the $i$th component. Then with slight abuse of notation, for each $d\ge 1$, we have a product action $\alpha:\mathbb{Z}_2^d\curvearrowright \mathcal{R}_\lambda$ that for $u=(u_1,\cdots, u_d)\in \mathbb{Z}_2^d$, 
\begin{align*}\alpha_{u}(A)=\prod_{u_i=1}\alpha_i(A).
\end{align*}
It follows that $\alpha$ is $\phi_\lambda$-preserving and commutes with its modular group. For a subset $I\subset [d]$, write $\alpha_{I}=\prod_{i\in I}\alpha_i$.
The spherical mean operator is
\[ M_k=\frac{1}{\binom{d}{k}}\sum_{I\subset [d]}\alpha_i.\]
Then by Theorem \ref{main3}, we deduce that for every $p>1$, any $A\in \mathcal{L}^+_p(\mathcal{R}_{\lambda}, \phi_\lambda)$ , there is $B\in \mathcal{L}^+_p(\mathcal{R}_{\lambda}, \phi_\lambda)$ such that for every $1\leq k\leq d$,
$$M_k(A) =\frac{1}{\binom{d}{k}}\sum_{I\subset [d]}\alpha_i(A)\le B\ \text{and}\ \|B\|_{\mathcal{L}_p(\mathcal{R}_{\lambda}, \phi_\lambda)}\leq C_{p}\|A\|_{\mathcal{L}_p(\mathcal{R}_{\lambda}, \phi_\lambda)}.$$
We note that the constant $C_p$ only depends on $p$ (independent of $d$) but the definition of $M_k$ depends on  $d$.


\appendix
\section{Haagerup's $L_p$-space}
In this appendix, we briefly review Haagerup noncommutative $L_p$-space and Haagerup reduction method. We refer to \cite{Terp,P2,HJX,GYZ} for more details on these topics.
\subsection{Haagerup noncommutative $L_p$-spaces}
 Let $\M\subset B(H)$ be a von Neumann algerba equipped with a distinguished normal faithful state $\phi$. For example, we can identify $\M\cong \pi_\phi(\M)$ via the GNS representation $\{\pi_\phi,H_\phi,\eta_\phi\}$ that
\[
\phi(x)=\langle \eta_\phi,\pi_\phi(x)\eta_\phi\rangle,\ x\in \M .
\]
The modular automorphism group $(\sigma_t^\phi)_{t\in\R}$ is given by
\[
\sigma_t^\phi:\M\to \M\  ,\ \sigma_t^\phi(x)=\Delta^{it}x\Delta^{-it},\qquad x\in \M ,
\]
where $x$ is identified with $\pi_\phi(x)$, $\Delta=S^*\overline S$ is the relative modular operator and the (closed) Tomita operator $\overline S$ is defined by
\[
S\big(\pi_\phi(x)\eta_\phi\big)=\pi_\phi(x^\ast)\eta_\phi.\qquad
\]

The crossed product
\[\mathcal{R}=\M\rtimes_{\sigma^\phi} \mathbb{R}\]
is the von Neumann algebra acting on $L_2(\mathbb{R},H)$, generated by the operators $\pi(x), x\in \M$, and $\la(s),s\in \mathbb{R}$, defined as follows: for all $\xi\in L_2(\mathbb{R},H)$ and $t\in \R$
\[
\pi(x)(\xi)(t)= \sigma_{-t}^{\varphi}(x)\xi(t),
\qquad
\la(s)(\xi)(t)=\xi(t-s).
\]
Note that $\pi$ is a normal faithful representation of $\M$ on $H\ten_2 L_2(\mathbb{R})\cong L_2(\mathbb{R},H)$ and $(\pi,\mathcal{L})$ gives a covariant representation such that
$$\pi(\sigma_t(x))=\mathcal{L}(t)\pi(x)\mathcal{L}(t)^*, x\in \M , t\in \mathbb{R}.$$
The dual action $\hat{\sigma}^\phi$ of $\mathbb{R}$ on $\mathcal{R}$ is  implemented by the unitary $W(t)$
\[\hat{\sigma}^\phi_t(x)=W(t)xW(t)^*\ , \ (W(t)\xi)(s)=e^{-ist}\xi(s)\ , \ x\in\mathcal{R}, \xi\in L_2(\mathbb{R},H),s,t\in\R. \]
One has
\[\hat{\sigma}^\phi_t(x)=x, \quad \hat{\sigma}^\phi_t(\mathcal{L}(s))=e^{-ist}\mathcal{L}(s), \ x\in \M, \ s,t\in \mathbb{R},\]
and $\M=\{x\in \mathcal{R}\  | \ \hat{\sigma}^\phi_t(x)=x\ ,  \forall t\in \mathbb{R}\}$ is exactly the invariant subspace under $\hat{\sigma}^\phi_t$. For any normal positive linear functional $\psi\in \M_*^+$, there is a dual weight $\widehat{\psi}$ on $\mathcal{R}$ determined by 
$$\widehat{\psi}(x)=\psi\Big[\int_{\R}\hat{\sigma}^\phi_t(x)dt\Big].$$

Recall that $\mathcal{R}$ is semifinite and admits a normal semifinite faithful trace $\tau$ satisfying
\[ \tau\circ \hat{\sigma}^\phi_t=e^{-t}\tau, \quad \forall t\in \mathbb{R}.\]
Then there exists a Radon-Nikodym derivative $D_\psi$ with respect to $\tau$ such that
\[\widehat{\psi}(x)=\tau(D_\psi^{\frac{1}{2}} x D_\psi^{\frac{1}{2}}), \quad x\in \mathcal{R}_+\ , \quad \hat{\sigma}^\phi_t(D_\psi)=e^{-t}D_\psi\ .\]
In particular, the Radon-Nikodym derivative $D_\phi$ of the dual weight $\widehat{\phi}$ of our distinguished n.s.f. weight $\phi$ is called the density operator of $\phi$.

Let $ L_0(\mathcal{R},\tau)$ denote the topological $*$-algebra of all measurable operators associated to $(\mathcal{R},\tau)$.
For $0<p\le \infty$, the Haagerup noncommutative $L_p$-space associated to $(\M,\phi)$ is then defined as
\[\mathcal{L}_p(\M,\phi)=\{ x\in L_0(\mathcal{R},\tau): \hat{\sigma}^\phi_t(x)=e^{-t/p}x, \ \forall t\in \mathbb{R}\}.\]
Recall that
$$\mathcal{L}_\infty(\M,\phi)=\M\ \text{and}\ \mathcal{L}_1(\M,\phi)=\M_*.$$
The latter equality are understood as follows: as mentioned above, there is a linear bijection
\[\psi\in \M_*^+ \longleftrightarrow D_\psi\in\mathcal{L}^+_1(\M,\phi).\]
This bijection further extends an identification
$\psi\in \M_* \leftrightarrow D_\psi\in \mathcal{L}_1(\M,\phi)$
with the module property $D_{x  \psi  y}=xD_{ \psi }y, x,y\in \M$.
Moreover, if $\psi=u|\psi|$ is its polar decomposition, then $D_\psi=u|D_\psi|=uD_{|\psi|}$. Thus using this linear bijection, the trace and $L_1$-norm on $\mathcal{L}_1(\M,\phi)$ is defined as
\[ \tr(D_\psi):=\psi(1)\ , \
\norm{D_\psi}{\mathcal{L}_1(\M,\phi)}:=\tr(|D_\psi|)=\tr(D_{|\psi|})=|\psi|(1)=\norm{\psi}{\M_*}.
\]
For $a\in L_0(\mathcal{R},\tau)$, we have the polar decomposition $a=u|a|$ and for $p\in [1,\infty)$
\[a\in \mathcal{L}_p(\M,\phi)\Longleftrightarrow |a|\in \mathcal{L}_p(\M,\phi)\Longleftrightarrow |a|^{p}\in \mathcal{L}_1(\M,\phi),\]
which leads to the $L_p$-norm of $a$ defined as
\[ \norm{a}{\mathcal{L}_p(\M,\phi)}=\tr(|a|^p)^{1/p}\ , \qquad \norm{a}{\infty}=\norm{a}{\M}.
\]
The trace ``$\tr$'' has the following tracial property:
\[\tr(ab)=\tr(ba),\]
given $a\in \mathcal{L}_p(\M,\phi),b\in \mathcal{L}_q(\M,\phi)$ with $1/p+1/q=1$, $ab,ba\in \mathcal{L}_1(\M)$. 
We collect the following facts about Haagerup noncommutative $L_p$-spaces.

	\begin{itemize}
	\item[(i)] If $\mathcal{M}$ is a semifinite von Neumann algebra equipped with a normal faithful trace $\tau$, then the Haagerup space $\mathcal{L}_p(\mathcal{M}, \tau)$ is isometrically isomorphic to the tracial $L_p$-space $L_p(\mathcal{M},\tau)$ introduced in Subsection~\ref{tracial}; see \cite[Chapter II]{Terp1}. 
    Indeed, in this case  $$\mathcal{R} = \mathcal{M} \overline{\otimes} L_\infty(\mathbb{R})\ ,\ D_\tau = \id \otimes \exp(\cdot)\ ,\  \hat{\sigma}^\tau_t = \id \otimes T_t,$$ where $T_t f(\cdot) = f(\cdot - t)$ is the shifting on $\mathbb{R}$. Any $x \in \mathcal{L}_p(\mathcal{M}, \tau)$ can be written as $x = x' \otimes \exp(\cdot/p)$ with $x' \in L_p(\mathcal{M},\tau)$, and moreover, $$\| x \|_{\mathcal{L}_p(\mathcal{M}, \tau)} = \| x \|_{L_p(\mathcal{M}, \tau)}:=\tau(|x'|^p)^\frac{1}{p}.$$	
	\item[(ii)] The Haagerup noncommutative $L_p$-spaces $\mathcal{L}_p(\M,\phi)$ is independent of the choice $\phi$ up to isometry. Namely, for two normal faithful states $\phi$ and $\psi$, the crossed product $\mathcal{R}_{\phi}:=\M\rtimes_{\sigma^\phi} \mathbb{R}$ and $\mathcal{R}_{\psi}\M\rtimes_{\sigma^\psi}\mathbb{R}$ are isomorphic by the Connes cocycle unitary $u_t=(D_{\phi}:D_{\psi})_t$ defined by the property
    \[ u_{s+t}=u_s \sigma^{\psi}(u_t). \]
This induces an isometry between the $L_p$ spaces $\mathcal{L}_p(\M,\phi)$ and $\mathcal{L}_p(\M,\psi)$. See  \cite[Chapter II]{Terp1}  and also \cite{GYZ} for details.
    \item[(iii)] Let $\mathcal{M}$ be a semifinite von Neumann algebra equipped with a normal faithful trace $\tau$, and let $\phi$ be a normal faithful state on $\mathcal{M}$. Then $\phi(x) = \tau(\sigma x)$ for $x \in \mathcal{M}$ for some positive operator  $\sigma\in L_1(\mathcal{M})$ with $\tau(\sigma) = 1$. The Kosaki's $L_p$-space $L_{p,\sigma}(\mathcal{M})$ is defined as the completion of $\mathcal{M}$ with respect to the norm
	\[
	\| x \|_{p,\sigma} = \tau( |\sigma^{\frac{1}{2p}} x \sigma^{\frac{1}{2p}} |^{p})^{\frac{1}{p}}=\norm{\sigma^{\frac{1}{2p}} x \sigma^{\frac{1}{2p}}}{p}.
	\]
	As noted in \cite[Proposition 2.4]{GYZ}, $L_{p,\sigma}(\mathcal{M})$ is isometric to Haagerup $L_p$-space that for $x\in \mathcal{M}$,
	\[
	 \| x \|_{p,\sigma}=\| D_\phi^{1/(2p)} x D_\phi^{1/(2p)} \|_{\mathcal{L}_p(\mathcal{M}, \phi)}.
	\]

\end{itemize}	 
Thanks to the module property $\mathcal{M}\mathcal{L}_p(\mathcal{M}, \phi)\mathcal{M}=\mathcal{L}_p(\mathcal{M}, \phi)$, the maximal $\mathcal{L}_p(\mathcal{M}, \phi; \ell_\infty)$ can be naturally defined as follows:
for a sequence $(x_n)$ in $\mathcal{L}_p(\mathcal{M}, \phi)$
$$\|(x_n)_{n}\|_{\mathcal{L}_p(\mathcal{M}, \phi; \ell_\infty)}=\inf\left\{\big\|a\big\|_{\mathcal{L}_{2p}(\mathcal{M}, \phi)}\sup_{n\in I}\big\|y_n\big\|_\infty\big\|b\big\|_{\mathcal{L}_{2p}(\mathcal{M}, \phi)}\right\},$$
where the infimum is taken over all possible factorizations $(x_n)_n=(ay_nb)_n$. We also denote the norm as $\|{\sup^+_{n}}x_n\|_{\mathcal{L}_p(\mathcal{M}, \phi)}$ or shortly $\|{\sup^+_{n}}x_n\|_{p, \phi}$. 
The column maximal $L_p$ space $\mathcal{L}_p(\mathcal{M}, \phi; \ell_\infty^c)$ is defined similarly.

By the above fact ii), it is clear from the definition that $\mathcal{L}_p(\mathcal{M}, \phi; \ell_\infty)$ for different states are isometric via Connes cocyle unitary.

	\subsection{Haagerup's reduction method}
	The idea of Haagerup's reduction method is to approximate a type III von Neumann algebra by finite ones.
	Let $\M$ be a von Neumann algebra equipped with a normal faithful state $\phi$. Let
	\[
	G=\bigcup_{n\in\N}2^{-n}\Z \subset \R ,
	\] be the discrete subgroups of dyadic numbers.
	Consider the crossed product for the action $\sigma^\phi:G\curvearrowright \M$
	\[
	\widehat{\M}=\M\rtimes_{\sigma^\phi} G \subset \M\ \overline\otimes\ B \!\big(\ell^2(G)\big)
	\]
	is the von Neumann subalgebra of $\M\otimes B(\ell_2(G))$ generated by the embeddings
	\begin{align*}
		&\pi:\M\to \M\otimes B(\ell_2(G)),\qquad
		\pi(a)=\sum_{g\in G}\sigma_{g^{-1}}(a)\otimes e_{g,g},\\
		&\lambda:G\to \M\otimes B(\ell_2(G)),\qquad
		\lambda(g)\big(a\otimes e_{h,h'}\big)=a\otimes e_{gh_1,h_2},\ \ \ x\in \M,\ h,h'\	\end{align*}
	where $e_{h,h'}$ represents the matrix unit in $B(\ell_2(G))$.
	Finite sums $\sum_g a_g \lambda(g)$ (with $a_g\in \M$) form a weak$^\ast$-dense $^\ast$-subalgebra of $\widehat{\M}$, and we identify $\M$ with $\pi(\M)\subset \widehat{\M}$.
	The state $\phi$ extends to a normal faithful state on $\widehat{\M}$ via
	\[
	\widehat{\phi} \Big(\sum_g a_g \lambda(g)\Big)=\phi(a_0),
	\]
	and the canonical $\phi$-preserving normal  conditional expectation $E_{\M}:\widehat{\M}\to \M$ is
	\[
	E_{\M} \Big(\sum_g a_g  \lambda(g)\Big)=a_0,\qquad \phi\circ E_{\M}=\widehat{\phi}.
	\]
	It follows that $\mathcal{L}_p(\M,\phi)\subset \mathcal{L}_p(\widehat{\M},\hat{\phi})$ isometrically as a subspace.
	The main object in Haagerup’s construction is an increasing family of centralizer subalgebras
	\[
	\widehat{\M}_n=\mathcal{M}^{\phi_n}:=\{x\in \widehat{\M}\mid \sigma_t^{\phi_n}(x)=x,\ \forall\,t\in\R\},
	\]
	associated with the state $\phi_n$ defined via a Radon–Nikod\'{y}m density w.r.t.~$\widehat{\phi}$:
	\[
	\phi_n(x)=\widehat{\phi}(e^{-a_n}x),\qquad a_n=-i\,2^{\,n}\,\text{Log} \big(\lambda(2^{-n})\big),
	\]
	where $\text{Log}$ is the principal branch with $0\le \text{Im}(\text{Log}(z))<2\pi$. Each $\widehat{\mathcal{M}}_n$ contains $\lambda(G)$ and admits a normal conditional expectation $E_{\M_n}:\widehat{\M}\to \widehat{\M}_n$ preserving $\widehat{\phi}$. Indeed, by the definition of $\phi_n$, the modular group $\sigma_t^{\phi_n}$ is $2^{-n}$ periodic. The explicit form (see \cite[Lemma 2.4]{HJX}) is given by
	\[E_{\widehat{\M}_n}=2^{n}\int_0^{2^{-n}}\sigma_t^{\phi_n} dt.\]
	
	Let $\tau_n=\phi_n/\phi_n(1)$ be the normalized tracial state on $\widehat{\mathcal{M}}_n$. The Haagerup's reduction theorem \cite[Theorem 2.1 \& 3.1]{HJX} is stated as follows.
	\begin{itemize}
		\item[i)] $(\widehat{\M}_n,\tau_n)$ is an increasing family of finite von Neumann algebra 
		such that $\bigcup_{n\ge 1}\widehat{\M}_n$ is weak$^\ast$-dense in $\widehat{\M}$;
        \item[ii)] For $n\ge 1$, there exists a normal $\hat{\phi}$-preserving conditional expectation $E_{n}:\widehat{\mathcal{M}}\to \widehat{\M}_n$
		satisfying  $ E_{\M}E_{n}=E_{\M}=E_{n}$;
		\item[iii)] $\mathcal{L}_p(\widehat{\mathcal{M}}_n, \hat{\phi}) \subset \mathcal{L}_p(\widehat{\mathcal{M}}, \hat{\phi})$ as subspaces and $\bigcup_{n\ge 1}\mathcal{L}_p(\widehat{\mathcal{M}}_n, \hat{\phi})$ is norm dense in $\mathcal{L}_p(\widehat{\mathcal{M}}, \hat{\phi})$. In particular, for any $x\in \mathcal{L}_p(\M,\phi)\subset \mathcal{L}_p(\widehat{\mathcal{M}}, \hat{\phi})$,
		$$\lim_{n\to \infty} E_n(x)=x \text{ in } \mathcal{L}_p(\widehat{\mathcal{M}}, \hat{\phi}).$$
	\end{itemize}
	Here $\mathcal{L}_p(\M,\phi)$ represents the Haargeup noncommutative $L_p$-space associted to $\phi$.

    It is proved in \cite[Section 7]{HJX} that the reduction extends to the maximal $L_p$-space $\mathcal{L}_p(\M,\phi;\ell_\infty)$ that 
    \begin{itemize}
    \item[i)] $\mathcal{L}_{p}(\M,\phi;\ell_\infty)$ and $\mathcal{L}_{p}(\widehat{\M}_n,\widehat{\phi};\ell_\infty)$ embed into $\mathcal{L}_{p}(\widehat{\M},\widehat{\phi};\ell_\infty)$ as a subspace isometrically \cite[Corollary 7.4]{HJX}.
    \item[ii)]  for a sequence $(x_i)_i\subset \mathcal{L}_{p}(\M,\phi)\subset \mathcal{L}_{p}(\widehat{\M},\widehat{\phi})$, 
    \[\lim_{n\to \infty}\norm{(E_n(x_i))_i}{\mathcal{L}_{p}(\widehat{\M}_n,\widehat{\phi};\ell_\infty)}=\norm{(x_i)_i}{\mathcal{L}_{p}(\M,\phi;\ell_\infty)}.\]
\end{itemize}

\medskip

\noindent {\bf Acknowledgement}. L. Gao~is partially supported by the National Natural Science Fundation of China (NSFC grant No.~12401163), and by the Department of Science and Technology of Hubei Province (Project No. 2025EHA041 and No. 2025AFA044). B. Xu is supported by National Natural Science Foundation of China (Grant No. 12501172), and Fundamental Research Funds for the Central Universities (Grant No. 20720250061).

\end{document}